\documentclass[reqno, oneside]{amsart}
\usepackage{amsfonts, amsmath, amsthm, amssymb}
\usepackage[shortlabels]{enumitem}
\usepackage[marginpar=2cm]{geometry}
\usepackage[utf8]{inputenc}
\usepackage{nicefrac}
\usepackage{tikz-cd}
\usepackage{multicol}
\RequirePackage{tasks}
\usepackage{float}
\usepackage{pifont}
\usepackage{adjustbox}
\usepackage{xspace}
\usepackage{etoolbox}
\usepackage{booktabs}
\usepackage[hidelinks]{hyperref}
\usepackage[capitalise,nosort]{cleveref}
\usepackage[colorinlistoftodos]{todonotes}
\usepackage{comment}
\usepackage{textgreek}
\usepackage{dutchcal}
\usetikzlibrary{decorations.pathmorphing}

\newtheorem{theorem}{Theorem}[section]
\newtheorem{lemma}[theorem]{Lemma}
\newtheorem{proposition}[theorem]{Proposition}
\newtheorem{corollary}[theorem]{Corollary}

\theoremstyle{definition}
\newtheorem{definition}[theorem]{Definition}
\newtheorem{remark}[theorem]{Remark}

\newtheorem{claim}{Claim}

\newtheorem{example}[theorem]{Example}

\crefformat{footnote}{#2\footnotemark[#1]#3}

\newcommand{\ReDeclareMathOperator}[2]{\let#1\relax\DeclareMathOperator{#1}{#2}}

\DeclareMathOperator{\at}{\mathcal{a{\mkern-2.5mu}t}}
\DeclareMathOperator{\pt}{\textit{pt}}

\DeclareMathOperator{\cons}{{\mathcal{Cons}}}

\DeclareMathOperator{\ats}{\mathcal{a{\mkern-2.5mu}t}_S}
\ReDeclareMathOperator{\int}{int}

\renewcommand{\O}{\mathcal O}
\newcommand{\C}{\mathcal C}
\newcommand{\B}{\mathcal B}
\newcommand{\F}{\mathcal F}
\newcommand{\sfO}{{\mathcal O}}
\newcommand{\sfC}{{\mathcal C}}
\newcommand{\LC}{{\mathcal {LC}}}

\newcommand{\s}{\mathcal s}

\ReDeclareMathOperator{\K}{K}
\ReDeclareMathOperator{\S}{Sat}
\ReDeclareMathOperator{\CS}{CSat}
\ReDeclareMathOperator{\KS}{KS}

\ReDeclareMathOperator{\LO}{CLC}
\ReDeclareMathOperator{\CSAT}{CSAT}
\ReDeclareMathOperator{\SFilt}{SFilt}
\ReDeclareMathOperator{\WLC}{WLC}
\ReDeclareMathOperator{\WLO}{WCLC}
\ReDeclareMathOperator{\GC}{AC}
\ReDeclareMathOperator{\P}{\mathcal{P}}
\ReDeclareMathOperator{\Pp}{\mathcal{P}_P}

\mathchardef\mhyphen="2D

\newcommand{\categoryname}[1]{\ensuremath{\mathbf{#1}}\xspace}
\newcommand{\DeclareCategory}[2]{\newcommand{#1}{\categoryname{#2}}}
\DeclareCategory{\MT}{MT}
\DeclareCategory{\MTD}{MT_D}
\DeclareCategory{\Skel}{Skel}
\DeclareCategory{\wMT}{PMT}
\DeclareCategory{\MTP}{MT_P}
\DeclareCategory{\PMTD}{PMT_D}
\DeclareCategory{\TDMTP}{TD MT_P}

\DeclareCategory{\TDMT}{TDMT}
\DeclareCategory{\ToMT}{T0MT}
\DeclareCategory{\TDMTD}{TDMT_D}
\DeclareCategory{\ToMTD}{T0MT_D}

\DeclareCategory{\SMT}{SMT}     
\DeclareCategory{\SMTD}{SMT_D}    
\DeclareCategory{\SMTP}{SMT_P}    

\DeclareCategory{\STDMT}{STDMT}    
\DeclareCategory{\STDMTP}{STDMT_P}    
\DeclareCategory{\STOMT}{ST0MT}    
\DeclareCategory{\STOMTD}{ST0MT_D}    
\DeclareCategory{\SwMT}{SwMT}
\DeclareCategory{\RMT}{RMT}
\DeclareCategory{\KMT}{KMT}

\DeclareCategory{\Frm}{Frm} 

\DeclareCategory{\FrmD}{Frm_D} 
\DeclareCategory{\TDFrm}{T_DFrm}
\DeclareCategory{\TDSFrm}{TD\mhyphen SFrm} 
\DeclareCategory{\TDSFrmD}{TD\mhyphen SFrm_D} 
\DeclareCategory{\pRaney}{pRaney}
\DeclareCategory{\BFrm}{BoolFrm}

\DeclareCategory{\CFrm}{CFrm}
\DeclareCategory{\SFrm}{SFrm}
\DeclareCategory{\SHFrm}{SHFrm}
\DeclareCategory{\NFrm}{NormFrm}
\DeclareCategory{\SNFrm}{SNormFrm}
\DeclareCategory{\HFrm}{HFrm}
\DeclareCategory{\RegFrm}{RegFrm}
\DeclareCategory{\SRegFrm}{SRegFrm}
\DeclareCategory{\CRegFrm}{CRegFrm}
\DeclareCategory{\SCRegFrm}{SCRegFrm}
\DeclareCategory{\KRegFrm}{KRegFrm}
\DeclareCategory{\Top}{Top} 
\DeclareCategory{\ToTop}{T0Top} 
\DeclareCategory{\TDTop}{TDTop} 
\DeclareCategory{\ToTopLC}{T0Top_{LC}} 
\DeclareCategory{\TopLC}{Top_{LC}}
\DeclareCategory{\TopS}{Top_{S}}
\DeclareCategory{\TDTopS}{TDTop_{S}}

\DeclareCategory{\Sob}{Sob} 
\DeclareCategory{\Bool}{BA}
\DeclareCategory{\BA}{Bool} 
\DeclareCategory{\DLat}{DLat}
\DeclareCategory{\HLat}{HLat}
\DeclareCategory{\Cons}{Cons}
\DeclareCategory{\HA}{Heyt}
\DeclareCategory{\cHA}{cHA}
\DeclareCategory{\IA}{Int} 
\DeclareCategory{\IAC}{Int_C}
\DeclareCategory{\SA}{Ess}
\DeclareCategory{\SAC}{Ess_C}
\DeclareCategory{\Reg}{Reg}
\DeclareCategory{\Set}{Set}
\DeclareCategory{\CABA}{CABA}

\renewcommand{\diamond}{\lozenge}
\let\ampersand\&
\renewcommand{\&}{\mathbin{\ampersand}}
\newcommand{\upset}{\mathord{\uparrow}}

\newcommand{\se}{\subseteq}

\setlist[enumerate]{font=\normalfont}

\tikzset{
  symbol/.style={
    draw=none,
    every to/.append style={
      edge node={node [sloped, allow upside down, auto=false]{$#1$}}}
  }
}

\makeatletter
\patchcmd{\@setaddresses}{\indent}{\noindent}{}{}
\patchcmd{\@setaddresses}{\indent}{\noindent}{}{}
\patchcmd{\@setaddresses}{\indent}{\noindent}{}{}
\patchcmd{\@setaddresses}{\indent}{\noindent}{}{}
\makeatother

\title{The Funayama envelope as the $T_D$-hull of a frame}
\author{G. Bezhanishvili, R. Raviprakash, A. L. Suarez, J. Walters-Wayland }
\address{New Mexico State University, Las Cruces, NM, USA}
\email{guram@nmsu.edu}
\email{prakash2@nmsu.edu}
\address{Coimbra University, Portugal}
\email{annalaurasuarez993@gmail.com}
\address{CECAT, Chapman University, Orange, CA, USA}
\email{jwalterswayland@gmail.com}
\subjclass[2020]{18F60; 18F70; 06D22; 06E25; 54E05; 54D10}
\keywords{Topology, frame,  interior algebra, proximity, $T_D$-separation}

\begin{document}

\begin{abstract}
    We introduce proximity morphisms between MT-algebras and show that the resulting category is equivalent to the category of frames. This is done by utilizing the Funayama envelope of a frame, which is viewed as the $T_D$-hull. Our results have some spatial ramifications, including a generalization of the $T_D$-duality of Banaschewski and Pultr. 
\end{abstract}

\maketitle

\tableofcontents

\section{Introduction}

Pointfree topology has its origins in the study of topological spaces where the lattice of open sets is taken as the core construct. Although this has been very
fruitful (see, e.g., \cite{Joh1982,PicadoPultr2012}), it has its own drawbacks because the language is often not expressive enough, especially to capture lower separation axioms (see, e.g., \cite{PP2021}). 
A more expressive formalism is obtained by 
describing a topological space via an interior operator on the powerset.
This approach goes back to Kuratowski \cite{Kur1922}, and has further been developed by McKinsey and Tarski \cite{MT1944} in the form of the theory of interior algebras; that is, boolean algebras equipped with an interior operator. As is manifested by the title of their article, {\em The Algebra of Topology}, they envisioned an algebraic formalism to reason about topology. 
This approach turned out to be very 
beneficial not only for topology, but also for the foundations of mathematics in general, and the connection between intuitionistic and modal logics in particular 
(see, e.g., \cite{RS1963,CZ1997,Esa2019}). 

As was demonstrated in \cite{BezhanishviliRR2023} (see also \cite{RS1963,Naturman1990}), especially important in topology are those interior algebras whose underlying boolean algebra is complete. Indeed, those can naturally be thought of as 
a pointfree generalization of spaces, and give rise to frames by taking the poset of open elements. 
These algebras were coined MT-algebras, in honor of McKinsey and Tarski. Equipping MT-algebras with an appropriate notion of morphism, 
we obtain the category $\MT$, 
and the open element functor $\O$ from $\MT$ to the category $\Frm$ of frames. Moreover, up to isomorphism, every frame arises as the frame of open elements of some MT-algebra, thanks to the well-known Funayama embedding of each frame into a complete boolean algebra \cite{funayama59}. We call this the Funayama envelope of a frame $L$ and denote it by $\F L$.
However, the assignment $L\mapsto \F L$ is not functorial: frame morphisms do not in general lift to complete lattice maps between their Funayama envelopes \cite[Example~4.4]{BezhanishviliRR2023}. To amend this, we weaken the morphisms between MT-algebras by basing the new morphisms on a proximity-like structure on the MT-algebra, reminiscent of de Vires proximity on a boolean algebra (see \cite{deV62,Bezh2010}). This modification enables us to obtain $\F$ as a functor from $\Frm$ to this new category $\bf MT_P$ of MT-algebras and proximity moprhisms. One of our main results establishes that the functors $\O$ and $\F$ yield an equivalence of these categories. 

As we will see, the Funayama envelope of a frame always satisfies $T_D$-separation. Spatially, this is the separation axiom of Aull and Thron \cite{aull62} stating that each point is locally closed.
The MT-version of it states that locally closed elements join-generate the MT-algebra. 
It appears that there is no satisfactory way to capture this separation in the language of frames.
We think of the Funayama envelope as the $T_D$-hull of a frame, thus providing a useful formalism to capture $T_D$-separation pointfreely, albeit in the language of MT-algebras rather than frames.
As a result, we obtain that each MT-algebra has a $T_D$-reflection, which also happens to be a coreflection.
This is explained by the fact that proximity morphisms are weak enough so that not all isomorphisms in this category are structure-preserving bijections. This, in particular, results in an equivalence of $\MTP$, its full subcategory $\TDMTP$ of $T_D$-algebras, and $\Frm$ (for the reader's convenience, all categories of interest are gathered together in tables at the end of the paper).

Our results have some spatial ramifications, including a further explanation and generalization of the $T_D$-duality of Banaschewski and Pultr \cite{banaschewskitd}. Indeed, we generalize their notion of a D-morphism of frames to that of a D-morphism between MT-algebras, and prove that the category $\STDMT$ of spatial $T_D$-algebras 
is a reflective subcategory of the category $\MTD$ of MT-algebras and D-morphsims, and is equivalent to the category $\TDTop$ of $T_D$-spaces. 
This yields a pointfree description of the $T_D$-coreflection of \cite[3.7.2]{banaschewskitd}, which is not expressible in the language of frames.
Another advantage of the MT-approach is that every MT-morphism between $T_D$-algebras is a D-morphism, which is in contrast with what happens in the setting of frames (where the category $\TDSFrmD$ of $T_D$-spatial frames with D-morphisms is not a full subcategory of $\Frm$). 

In the $T_D$-duality of \cite{banaschewskitd},  
only D-morphisms are captured as duals of continuous maps whereas we are able to capture all frame morphisms, as well as the corresponding proximity morphisms. This is done by introducing the notion of sober maps (that is, continuous maps from one topological space to the soberification of another), thus obtaining  a more general duality for $T_D$-algebras that subsumes the $T_D$-duality for frames. The latter is the restriction of a new 
duality between the categories $\TopS$ of topological spaces and sober maps and $\SMTP$ of spatial MT-algebras and proximity morphisms.

The ability to describe the $T_D$-hull of a frame provides further evidence that this alternate pointfree approach to topology is highly beneficial by affording sufficiently expressive power to capture lower separation axioms, which have been elusive in locale theory.

\section{Preliminaries} \label{sec: prelims}

In this section we briefly review some well-known facts about frames and MT-algebras that we will use in the rest of the paper. 

\subsection{Frames and co-frames}\label{frames and coframes}
We recall that a complete lattice $L$ is a {\em frame} if it satisfies the join-infinite distributive law
$$
a\wedge\bigvee S=\bigvee\{a\wedge s\mid s\in S\},
$$
and a {\em co-frame} if it satisfies the meet-infinite distributive law
$$
a\vee\bigwedge S=\bigwedge\{a\vee s\mid s\in S\}
$$
for all $a\in L$ and $S\subseteq L$. 
A {\em frame morphism} is a map between frames preserving arbitrary joins and finite meets, and a {\em co-frame morphism} is defined dually. As usual, we let $\Frm$ denote the category of frames and frame morphisms.

A typical example of a frame is the complete lattice $\Omega X$ of open sets of a topological space $X$, and of a co-frame the complete lattice $\Gamma X$ of closed sets of $X$. The assignment $X \mapsto \Omega X$ is the object part of the functor $\Omega:\Top\to \Frm^{\mathrm{op}}$, which sends each continuous map $f:X\to Y$ to the preimage map $f^{-1}:\Omega Y\to\Omega X$. The functor $\Omega$ has a right adjoint, namely the 
functor $\pt:\Frm^{\mathrm{op}}\to \Top$, which maps each frame $L$ to the space of its {\em points} (completely prime filters) with the topology given by $\sigma_L[L]$, where $\sigma_L(a) = \{P \in \pt L \mid a \in P \}$ for each $a\in L$.
The functor $\pt$ sends each frame morphism $f:L\to M$ to the continuous map $f^{-1}:\pt M \to\pt L$.

A frame $L$ is \emph{spatial} provided points of $L$ separate non-comparable elements of $L$, and a space $X$ is {\em sober} if each irreducible closed set is the closure of a unique point. The adjunction $\Omega \dashv \pt$ restricts to an equivalence between the full subcategories $\SFrm$ of spatial frames and $\Sob$ of sober spaces (see \cite[Ch.~II]{PicadoPultr2012} for details): 
\[\begin{tikzcd}
	\Top && {\Frm^{\mathrm{op}}} \\
	\\
	\Sob && {\SFrm^{\mathrm{op}}}
	\arrow["\Omega",shift left=2, from=1-1, to=1-3]
	\arrow["\bot"',"\pt", shift left=2, from=1-3, to=1-1]
	\arrow["\mathrm{full}", hook, from=3-1, to=1-1]
	\arrow["\cong",<-> ,from=3-1, to=3-3]
   \arrow["\mathrm{full}"', hook, from=3-3, to=1-3]
\end{tikzcd}\]

More important for our purposes is the $T_D$-duality of Banaschewski and Pultr \cite{banaschewskitd}. We recall \cite{aull62} that a topological space $X$ is a {\em $T_D$-space} if each point $x \in X$ is locally closed (closed in some open neighborhood of $x$). Let $\TDTop$ be the full subcategory of $\Top$ consisting of $T_D$-spaces.\footnote{In \cite[3.1]{banaschewskitd} this category is denoted by $\Top_D$.} If $f:X\to Y$ is a continuous map between $T_D$-spaces, then $f^{-1}:\Omega Y\to\Omega X$ has an extra property. To describe it, we recall that an element $a$ of a poset $P$ is covered by another element $b$ if $a < b$ and from $a \leq x \leq b$ it follows that $a = x$ or $x = b$. In this case we write $a \lessdot b$. An element $a$ is said to be \emph{covered} if $a \lessdot b$ for some $b$.

\begin{definition} \label{def: slicing filter}
     \cite[ Sec.~2.6]{banaschewskitd} A filter $F$ of a frame $L$ is \emph{slicing} if $F$ is prime and there exist $b \in F$ and $a \not \in F$ with $a \lessdot b$.
\end{definition}

The \emph{$T_D$-spectrum} of a frame $L$ is defined to be the collection $\pt_D L$ of slicing filters of $L$, topologized by setting the opens to be the elements of the form $\delta(a) = \{x \in \pt_D L \mid a \in x\}$. As was shown in \cite[Lem.~2.6.1]{banaschewskitd}, $\pt_D L$ is a subspace of $\pt L$. The dual frame morphisms of continuous maps between $T_D$-spaces have the following extra property: 

\begin{definition} \label{def: D morphism}
\cite[ Sec.~3.1]{banaschewskitd}
    A frame morphism $f : L \to M$ is a \emph{D-morphism} provided $f^{-1}(F)$ is a slicing filter of $L$ for each slicing filter $F$ of $M$.  
\end{definition}

Let $\FrmD$ denote the category of frames and D-morphisms between them. Clearly $\FrmD$ is a wide subcategory of $\Frm$.
A frame $L$ is said to be \emph{$T_D$-spatial} if it is isomorphic to $\Omega X$ for some $T_D$-space $X$. Let $\TDSFrmD$ be the full subcategory of $\FrmD$ determined by these objects. We then have (see \cite[Prop.~3.5.1]{banaschewskitd}):

\begin{theorem} [$T_D$-duality] \label{thm: TD duality of BP}
    $(\Omega,\pt_D)$ is an adjunction between $\TDTop$ and $\FrmD^{\mathrm{op}}$, which restricts to an equivalence between $\TDTop$ and $\TDSFrmD^{\mathrm{op}}$.
\[\begin{tikzcd}
	\TDTop && {\FrmD^{\mathrm{op}}} \\
	\\
	\TDTop && {\TDSFrmD^{\mathrm{op}}}
	\arrow["\Omega",shift left=2, from=1-1, to=1-3]
	\arrow["\bot"',"\pt_D", shift left=2, from=1-3, to=1-1]
	\arrow["=",rightarrow, from=3-1, to=1-1]
    \arrow[leftarrow, from=3-1, to=1-1]
	\arrow["\cong",<->, from=3-1, to=3-3]
  \arrow["\mathrm{full}"', hook, from=3-3, to=1-3]
\end{tikzcd}\]
 \end{theorem}
The $\pt$ and $\pt_D$ functors are in general different, even for $T_D$-spatial frames.
Hence, 
the $T_D$-duality is not a restriction of the $\Omega\dashv \pt$ adjunction. But the functor $\Omega$ is the same in both cases, thus we do have the following commutative square: 

\[\begin{tikzcd}
	\Top && {\Frm^{\mathrm{op}}} \\
	\\
	\TDTop && {\TDSFrmD^{\mathrm{op}}}
	\arrow["\Omega",from=1-1, to=1-3]
	\arrow["\mathrm{full}", hook, from=3-1, to=1-1]
	\arrow["\cong",<->, from=3-1, to=3-3]
	\arrow["\mathrm{nonfull}"', hook,dashed, from=3-3, to=1-3]
\end{tikzcd}\]

We emphasize that not all frame morphisms between $T_D$-spatial frames are D-morphisms. The next example illustrates this. 

\begin{example}\label{e:notDmorphism}
    Let $X:=\{*\}$ be a singleton space and $Y$ the natural numbers equipped with the Alexandroff topology (where opens are precisely the upper sets). It is easy to see that both $X$ and $Y$ are $T_D$-spaces. Moreover, $\Omega X$ is isomorphic to the two-element boolean algebra $2 = \{ 0,1 \}$ and $\Omega  Y$ is isomorphic to $(\omega+1)^{\mathrm{op}}$. Define $f:\Omega  Y\to \Omega X$ by $f(a)=1$ iff $a\neq 0$. It is straightforward to check that $f$ is a frame morphism. Furthermore, $F:=\{1\}$ is a slicing filter in $\Omega X$, but $f^{-1}(F)=\Omega Y\setminus\{\varnothing\}$ is not a slicing filter in $\Omega Y$. Thus, $f$ is not a D-morphism.
    \end{example}

    \begin{remark}
        In the above example,
       no continuous map between the spaces $X$ and $Y$ can give rise to $f : \Omega Y \to \Omega X$ since otherwise $*$ would have to be mapped to a point whose open neighborhoods are all nonempty opens of $Y$, and such a point does not exist in $Y$. In fact, all frame morphisms between $T_D$-spatial frames that come from continuous maps between $T_D$-spaces are D-morphisms (as will be evident from \cref{thm: MT duality for T0 and TD 2,cor: TD frames and MT} below).
    \end{remark}

\subsection{Interior algebras and MT-algebras}\label{Interior algebras and MT-algebras}

The following definitions go back to McKinsey and Tarski \cite{MT1944} (see also \cite{RS1963,Esa2019}).

\begin{definition} 
An {\em interior operator} on a boolean algebra $A$ is a unary function $\square : A\to A$ satisfying Kuratowski's axioms 
for all $a,b\in A$:
\begin{itemize}
    \item $\square 1=1$.
    \item $\square(a \wedge b) = \square a \wedge \square b$. 
    \item $\square a \leq a$. 
    \item $\square a \leq \square \,\square a$. 
\end{itemize}
An {\em interior algebra} is a pair $(A,\square)$ with $A$ a boolean algebra and $\square$ an interior operator on~$A$. 
\end{definition}

Recall (see, e.g., \cite[Sec.~III.3]{RS1963}) that a {\em morphism of interior algebras} is a boolean homomorphism $f:A\to B$ such that $f(\Box a) = \Box f(a)$ for each $a\in A$. We will be interested in the following weaker condition: $f(\Box a) \le \Box f(a)$ for each $a\in A$. Such morphisms have been studied in the literature under different names: continuous morphisms \cite{Ghilardi2010}, stable morphisms \cite{BBI2016}, or semi-homomorphisms \cite{BMM08}. For the purposes of this paper, we will follow \cite{Ghilardi2010} in calling them {\em  continuous morphisms}. 

\begin{definition}\ \label{def: IA and Int}
    \begin{enumerate}
        \item Let $\IA$ be the category of interior algebras and interior algebra morphisms.
        \item Let $\IAC$ be the category of interior algebras and continuous morphisms.
    \end{enumerate}
\end{definition}

Clearly $\IA$ is a wide subcategory of $\IAC$, and in both categories, compostion is function composition and identity morphisms are identity functions.
\begin{definition} 
Let $A$ be an interior algebra. 
\begin{enumerate}
\item An element $a \in A$ is {\em open} if $a= \square a$. 
\item An element $a\in A$ is {\em closed} if $a=\diamond a$ where $\diamond a = \neg \square \neg \centering a$.
\item An element $a \in A$ is \emph{locally closed} if $a= \square b \wedge \diamond c$ for some $b,c \in A$.\footnote{Equivalently, $a$ is locally closed provided $a=u\wedge\Diamond x$ for some $u\in\O A$.\label{footnote 2}} 
\end{enumerate}
Let $\sfO A$, $\sfC A$, and $\LC A$ be the collections of open, closed, and locally closed elements of $A$, respectively.
\end{definition}

Observe that $\sfO A$ is a bounded sublattice of $A$ and $\square:A\to\sfO A$ is right adjoint to the inclusion $\sfO A\hookrightarrow A$, yielding that $\sfO A$ is a Heyting algebra (see, e.g., \cite[Sec.~2.5]{Esa2019}). Similarly, $\sfC A$ is a bounded sublattice of $A$ and $\diamond:A\to\sfC A$ is left adjoint to the inclusion $\sfC A\hookrightarrow A$, yielding that $\sfC A$ is a co-Heyting algebra. We point out that the  implication on $\sfO A$ is given by $u \rightarrow v = \square(\neg u \vee v)$ for all $u , v \in \sfO A$, and the co-implication on $\sfC A$ by $c \leftarrow d = \diamond(d \wedge \neg c)$ for all $c , d \in \sfC A$. Moreover, $\LC A$ is closed under finite meets, and closing $\LC A$ under finite joins gives the boolean subalgebra of $A$ generated by $\O A$ (or $\C A$). 

\begin{definition}\ \label{defn: MT algebra}
\begin{enumerate}
\item An interior algebra $(A,\square)$ is a {\em McKinsey-Tarski algebra} or an {\em MT-algebra} if $A$ is a complete boolean algebra.
\item  An {\em MT-morphism} between MT-algebras $M$ and $N$ is a complete boolean homomorphism $h : M\to N$ such that $h(\square a) \leq \square h(a)$ for each $ a \in M$.
\item Let $\MT$ be the category of MT-algebras and MT-morphisms. 
\end{enumerate}
\end{definition}

Since each MT-algebra $M$ is complete and $\Box:M\to\sfO M$ is right adjoint to the inclusion, $\sfO M$ is a subframe of $M$. In fact, we can equivalently think of MT-algebras as pairs $(A,L)$ where $A$ is a complete boolean algebra and $L$ is a subframe of $A$. Then the interior operator on $A$ is given by 
\[
\Box a = \bigvee \{ b \in L \mid b \le a \}.
\]
Moreover, if $f:M \to N$ is an MT-morphism, then its restriction $f|_{\sfO M}:\sfO M \to\sfO N$ is a frame morphism that sends identity morphisms to identity morphisms and respects composition. We thus obtain:

\begin{theorem} {\em \cite[Thm.~3.10]{BezhanishviliRR2023}} \label{thm: O functor} 
The assignment $M \mapsto \sfO M$ and $f \mapsto f|_{\sfO M}$ yields a functor $\O : \MT \to \Frm$.
\end{theorem}

A typical example of an MT-algebra is  the powerset algebra $(\P X, \Box)$ of a topological space $X$, where $\Box$ is the interior operator on $X$. The assignment $X \mapsto \P X$ extends to a functor ${\P : \Top \to \MT^{\mathrm{op}}}$, where a continuous map $f:X\to Y$ is sent to the MT-morphism $f^{-1}:\P Y \to \P X$. To define a functor in the opposite direction, for an MT-algebra $M$, let $\at M$ be the set of atoms of $M$. For $a\in M$, define 
\[
\eta_M(a) = \{ x \in \at M \mid x \le a \}.
\]
Then $\{ \eta_M(a) \mid a \in M \}$ is a topology on $\at M$, so $\P \at M$ is an MT-algebra and $\eta_M: M \to \P\at M$  is an onto MT-morphism. Moreover, if $f:M\to N$ is an MT-morphism, then it has a left adjoint (since it is a complete boolean homomorphism). The restriction of the left adjoint is then a well-defined continuous map $f^* : \at N \to \at M$. This defines a functor $\at : \MT^{\mathrm{op}} \to \Top$, which is right adjoint to $\P$. 

We call $M$ {\em spatial} provided $\eta_M:M\to\P\at M$ is one-to-one (in which case it is an isomorphism of MT-algebras). 
Let $\SMT$ be the full subcategory of $\MT$ consisting of spatial MT-algebras. For each $X \in \Top$, let $\varepsilon_X:X \to \at\P X$ be given by $\varepsilon_X(x)=\{x\}$. Then $\varepsilon_X$ is a homeomorphism and we have (see \cite[Thm.~3.22]{BezhanishviliRR2023}):

\begin{theorem} [MT-duality] \label{MT duality}
$(\P,\at)$ is an adjunction between $\Top$ and $\MT^{\mathrm{op}}$ whose unit is given by $\varepsilon:1_{\Top} \to \at \circ \P$ and counit by $\eta:\P \circ \at \to 1_{\MT}$. This adjunction restricts to an equivalence between $\Top$ and $\SMT^{\mathrm{op}}$.
\[\begin{tikzcd}
	\Top && {\MT^{\mathrm{op}}} \\
	\\
	\Top && {\SMT^{\mathrm{op}}}
	\arrow["\P",shift left=2, from=1-1, to=1-3]
	\arrow["\bot"',"\at", shift left=2, from=1-3, to=1-1]
	\arrow["=",  from=3-1, to=1-1]
    \arrow[  from=1-1, to=3-1]
	\arrow["\cong",<->,from=3-1, to=3-3]
    \arrow["\mathrm{full}"', hook, from=3-3, to=1-3]
\end{tikzcd}\]
\end{theorem}

We conclude this preliminary section by recalling the MT-algebra analogues of $T_D$ and $T_0$-spaces. 
An element of an MT-algebra $M$ is \emph{saturated} if it is a meet from $\O M$. Let $\mathcal{S} M$ be the collection of saturated elements of $M$. We call $a\in M$ {\em weakly locally closed} if $a = s \wedge c$ where $s\in\mathcal{S} M$ and $c\in \C M$. Let $\mathcal{WLC} M$ be the collection of weakly locally closed elements of $M$. 

\begin{definition} \label{def: T0 and TD algebras}
    An MT-algebra $M$ is said to be a {\em $T_D$-algebra} if $\mathcal{LC} M$ join-generates $M$ and a {\em $T_0$-algebra} if $\mathcal{WLC} M$ join-generates $M$.
\end{definition}

Let $\TDMT$ be the full subcategory of $\MT$ consisting of $T_D$-algebras, let $\STDMT$ be the full subcategory of $\TDMT$ consisting of spatial $T_D$-algebras, and define $\ToMT$ and $\STOMT$ similarly. Let also $\TDTop$ be the full subcategory of $\Top$ consisting of $T_D$-spaces, and define $\ToTop$ similarly. Then  MT-duality restricts to yield the following (see \cite[Thms.~5.7, 6.4]{BezhanishviliRR2023}):

\begin{theorem}\ \label{thm: MT duality for T0 and TD}
    \begin{enumerate}[ref=\thetheorem(\arabic*)]
        \item The adjunction $(\P,\at)$ restricts to an adjunction between {\em $\ToTop$} and {\em $\ToMT^{\mathrm{op}}$}, yielding an equivalence between {\em $\ToTop$} and {\em $\STOMT^{\mathrm{op}}$}.\label[theorem]{thm: MT duality for T0 and TD 1}
        \item The adjunction $(\P,\at)$ further restricts to an adjunction between {\em $\TDTop$} and {\em $\TDMT^{\mathrm{op}}$}, yielding an equivalence between {\em $\TDTop$} and {\em $\STDMT^{\mathrm{op}}$}.\label[theorem]{thm: MT duality for T0 and TD 2}
    \end{enumerate}
\end{theorem}

Putting \cref{thm: TD duality of BP,thm: MT duality for T0 and TD 2} together, we conclude: 

\begin{corollary}\label{cor: TD frames and MT}
    {\em $\TDSFrmD$} is equivalent to {\em $\STDMT$}.
\end{corollary}

As we pointed out after \cref{thm: TD duality of BP}, the inclusion $\TDSFrmD \hookrightarrow \Frm$ is not full. By contrast, the inclusion $\STDMT \hookrightarrow \MT$ is full. Moreover, while the $T_D$-duality for frames is not a restriction of the $\Omega\dashv \pt$ adjuntion (since  $\pt_D$ is not in general the same as $\pt$), the $T_D$-duality for MT-algebras is obtained by restricting the adjunction $\P\dashv \at$.
This is summarized in the two diagrams below: 

\begin{center}
\begin{minipage}{0.4\textwidth}
\[\begin{tikzcd}
	\Top && {\Frm^{\mathrm{op}}} \\
	\\
	\TDTop && {\TDSFrmD^{\mathrm{op}}}
	\arrow["\Omega",from=1-1, to=1-3]
	\arrow["\mathrm{full}", hook, from=3-1, to=1-1]
	\arrow["\cong",<->,from=3-1, to=3-3]
   \arrow["\mathrm{nonfull}"', hook,dashed, from=3-3, to=1-3]
\end{tikzcd}\]    
\end{minipage}
\begin{minipage}{0.4\textwidth}
\[\begin{tikzcd}
	\Top && {\MT^{\mathrm{op}}} \\
	\\
	\TDTop && {\STDMT^{\mathrm{op}}}
	\arrow["\P",shift left=2, from=1-1, to=1-3]
	\arrow["\bot"',"\at", shift left=2, from=1-3, to=1-1]
	\arrow["\mathrm{full}",hook, from=3-1, to=1-1]
  \arrow["\cong",<->, from=3-1, to=3-3]
  \arrow["\mathrm{full}"', hook, from=3-3, to=1-3]
\end{tikzcd}\]
\end{minipage}
\end{center}

\section{Proximity morphisms between MT-algebras}\label{sec: proximity morphisms}

In this section, we show that each MT-algebra can be equipped with a proximity relation, which is a weakening of a de Vries proximity on a boolean algebra \cite{deV62,Bezh2010}. This gives rise to a new category $\MTP$ of MT-algebras and proximity morphisms between them. In \cref{sec: Funayama}, it will be shown that this category is equivalent to $\Frm$.

Given boolean algebras $A,B$, with $B$ a subalgebra of $A$, we define a binary relation $\prec_B$ on $A$ by
\[
a \prec_B c \Longleftrightarrow \exists b \in B : a \le b \le c.
\]
It is straightforward to verify that this relation satisfies the following conditions:
\begin{enumerate}
    \item[(S1)] $1\prec_B 1$;
    \item[(S2)] $a\prec_B c$ implies $a\leq c$;
    \item[(S3)] $a\leq a'\prec_B c'\leq c$ implies $a\prec_B c$;
    \item[(S4)] $a\prec_B c,d$ implies $a\prec_B c\wedge d$;
    \item[(S5)] $a\prec_B c$ implies $\neg c\prec_B \neg a$;
    \item[(S6)] $a\prec_B c$ implies that there is $b\in B$ with $a\prec_B b\prec_B c$.
\end{enumerate}

\begin{remark}\label{rem: de Vries proximity}
   The above are standard proximity axioms on a boolean algebra, except (S6) is a strengthening of the usual in-betweenness axiom. However, in general, $\prec_B$ is not a de Vries proximity on $A$ since it is not necessarily the case that $a=\bigvee \{c\in A\mid c\prec_B a\}$. In fact, $\prec_B$ is a de Vries proximity on $A$ if and only if $B$ join-generates $A$.
\end{remark}

In our considerations, $A$ will be an MT-algebra and $B$ will be the boolean subalgebra of $A$ generated by $\sfO A$. We recall that in the powerset algebra of a topological space, the elements of the boolean subalgebra generated by the frame of opens are exactly the finite unions of locally closed subsets, and are called constructible sets (see, e.g., \cite[p.~94]{Har1977}). In analogy: 

\begin{definition}
     An element $a$ of an MT-algebra $M$ is \emph{constructible} provided $a$ is a finite join from $\LC M$. Let $\cons M$ be the set of constructible elements of $M$.
\end{definition}

Note that $\cons M $ is the boolean subalgebra of $M$ generated by $\sfO  M$, and thus one can consider the associated binary relation, $\prec_{\cons M}$. To simplify notation, we omit the subscript.

\begin{definition} \label{def: proximity}
    Let $M$ be an MT-algebra. An element $a$ is \emph{constructibly below} $b$, or ``\emph{cons-below}", if 
    $a \prec b$ for the binary relation associated with $\cons M$.
    \end{definition}

Interestingly, the cons-below relation on an MT-algebra $M$ is a de Vries proximity precisely when $M$ is a $T_D$-algebra:

\begin{lemma}\label{lem: proximity and TD}
    For any MT-algebra $M$, the con-below relation is a de Vries proximity on $M$ iff $M$ is a $T_D$-algebra.
\end{lemma} 

\begin{proof}
By \cref{rem: de Vries proximity}, the cons-below relation is a de Vries proximity on $M$ iff $\cons M$ join-generates $M$. Since each element of $\cons M$ is a finite join from $\LC M$, the latter condition is equivalent to $M$ being a $T_D$-algebra.
\end{proof}

Next, in analogy with de Vries algebras, we define proximity morphisms between MT-algebras.

\begin{definition} \label{def: proximitymor}
    For $M, N \in \MT$, a map $f:M \to N$ is a \emph{proximity morphism} provided the following conditions are satisfied:
\begin{enumerate}[label=\upshape(P\arabic*)]
    \item $f| _{\sfO(M)} : \sfO M\to\sfO N$ is a frame morphism.\label{def: proximitymor 1}
    \item $f(a \wedge b)=f(a) \wedge f(b)$ for each $a,b \in M$.\label{def: proximitymor 2}
    \item $f(\bigvee S)=\bigvee\{f(s) \mid s \in S\}$ for each finite $S \subseteq \LC M$.\label{def: proximitymor 3}
    \item $f(a)=\bigvee\{ f(x) \mid x \in \LC M, \, x \leq a\}$ for each $a \in M$.\label{def: proximitymor 4}
\end{enumerate}
\end{definition}

\begin{remark} \label{rem: prox equiv}
    Since each element of $\cons M$ is a finite join from $\LC M$, \ref{def: proximitymor 4} is equivalent to \[
    f(a)=\bigvee\{ f(b) \mid b \in \cons M, \, b \leq a\} \ \mbox{ for each } a \in M.
    \]
\end{remark}

\begin{lemma}\label{l:proxmorproperties}
    Let $f:M \to N$ be a proximity morphism between MT-algebras.
    \begin{enumerate}[ref=\thelemma(\arabic*)]
    \item $f(\neg x)=\neg f(x)$ for each $x \in \sfO M \cup \sfC M$.\label[lemma]{l:proxmorproperties 1}
    \item $f| _{\sfC M} : \sfC M\to\sfC N$ is a co-frame morphism.\label[lemma]{l:proxmorproperties 2}
        \item If $x \in \LC M$ then $f(x) \in \LC N$. \label[lemma]{l:proxmorproperties 3}\label[lemma]{lem: lc to lc}
         \item $f|_{\cons M}:\cons M \to \cons N$ is a boolean homomorphism.\label[lemma]{l:proxmorproperties 4}
    \end{enumerate}
\end{lemma}
\begin{proof}
(1) Let $x\in\sfO M \cup \sfC M$. Then $\neg x \in \sfC M \cup \sfO M$. Thus, since $\sfO M \cup \sfC M \subseteq \LC M$, by \ref{def: proximitymor 3},
$f(x) \vee f(\neg x)=f(x \vee \neg x)=f(1)=1$. Moreover, by \ref{def: proximitymor 2},
 $f(x) \wedge f(\neg x) = f(x \wedge \neg x) = f(0) = 0$. Therefore, $f(\neg x)=\neg f(x)$. 

(2) Since $\C M \subseteq \LC M$, the restriction $f| _{\sfC M}$ preserves finite joins by\ref{def: proximitymor 3}. We show that it preserves arbitrary meets. Let $S \subseteq \sfC M$. Then $\neg s \in \sfO M$ for each $s\in S$, so $\bigvee \{ \neg s \mid s \in S \} \in \sfO M$. Therefore, by \ref{def: proximitymor 1} and (1), 
\begin{eqnarray*}
f\left(\bigwedge S\right) &=& f\left(\neg \bigvee \{ \neg s \mid s\in S \}\right) = \neg f\left(\bigvee \{ \neg s \mid s \in S \}\right) \\
&=& \neg \bigvee \{ f(\neg s) \mid s \in S \} = \neg \bigvee \{ \neg f(s) \mid s \in S \} \\
&=& \bigwedge f[S].
\end{eqnarray*}
Thus, $f| _{\sfC M} : \sfC M\to\sfC N$ is a co-frame morphism. 

(3) This follows from \ref{def: proximitymor 1}, \ref{def: proximitymor 2}, and (2).

(4) Since each element of $\cons M$ is a finite join from $\LC M$, it follows from \ref{def: proximitymor 3} and (3) that $f|_{\cons M}$ is well defined. Moreover, by \ref{def: proximitymor 1}--\ref{def: proximitymor 3}, it is a bounded lattice homomorphism, and thus a boolean homomorphism.
\end{proof}

\begin{lemma} \label{lem: eqv morphisms}
    Let $f : M \to N$ be a map between MT-algebras satisfying \ref{def: proximitymor 1}, \ref{def: proximitymor 2}, and \ref{def: proximitymor 4}. The following are equivalent:
    \begin{enumerate}
    [ref=\thelemma(\arabic*)]
        \item $f$ satisfies \ref{def: proximitymor 3}; that is, $f$ is a proximity morphism.
        \item $a_1 \prec b_1$ and $a_2 \prec b_2$ imply $f(a_1 \vee a_2) \prec f(b_1) \vee f(b_2)$ for each $a_i,b_i \in M$.
        \item $a \prec b$ implies $\neg f(\neg a) \prec f(b)$ for each $a,b \in M$. \label[lemma]{lem: eqv morphisms 3}
    \end{enumerate}
\end{lemma}

\begin{proof}
    It is sufficient to prove that (1)$\Leftrightarrow$(2) since (2)$\Leftrightarrow$(3) follows from \cite[Lem.~2.2]{Bezh2012} and \cite[Prop.~7.4]{BezhJohn2014}.

    (1)$\Rightarrow$(2):  
    Suppose $a_1 \prec b_1$ and $a_2 \prec b_2$. Then there exist finite $S_1,S_2\subseteq\LC M$ such that $a_1 \le \bigvee S_1 \le b_1$ and $a_2 \le \bigvee S_2 \le b_2$. Therefore, $a_1 \vee a_2 \le \bigvee S_1 \vee \bigvee S_2 \le b_1 \vee b_2$. By \ref{def: proximitymor 2}, $f$ is order preserving. Thus, by (1), 
    \[
    f(a_1\vee a_2) \le f\left(\bigvee S_1 \vee \bigvee S_2\right) = \bigvee f[S_1] \vee \bigvee f[S_2] \le f(b_1) \vee f(b_2).
    \]
    Consequently, $f(a_1 \vee a_2) \prec f(b_1) \vee f(b_2)$ since $\bigvee f[S_1] \vee \bigvee f[S_2]\in\cons N$ by \cref{l:proxmorproperties 3}.

    (2)$\Rightarrow$(1): Let $x,y\in\LC M$. Since $f$ is order preserving, $f(x)\vee f(y)\le f(x\vee y)$. For the reverse inequality,  $x\prec x$ and $y\prec y$. Therefore, by (2), $f(x\vee y) \prec f(x) \vee f(y)$, and hence $f(x\vee y)\le f(x)\vee f(y)$. 
    \end{proof}

We next show that the MT-algebras and proximity morphisms between them form a category, however neither the composition is usual function composition nor the identity morphisms are identity functions. The composition of proximity morphisms between MT-algebras is defined as for de Vries algebras:

\begin{definition}
    For proximity morphisms $f:M_1\to M_2$ and $g:M_2\to M_3$, define 
    \[
(g \star f)(a)=\bigvee \{g(f(x)) \mid x \in \LC M_1, \, x \leq a \}.
\]
\end{definition}

It is immediate from the above definition that if $x \in \LC M_1$ then $(g \star f)(x)=(g \circ f)(x)$.

\begin{lemma} \label{lem: composition}
    Let $f:M_1\to M_2$, $g:M_2\to M_3$, and $h:M_3\to M_4$ be proximity morphisms. For each $a\in M_1$, we have 
    \[
    ((h \star g) \star f)(a) = \bigvee \{h(g(f(x))) \mid x \in \LC M_1, \, x \leq a \} = (h \star (g \star f))(a).
    \]
    \end{lemma}

\begin{proof}
Let $a \in M_1$. Then
\begin{align*}
    ((h \star g) \star f)(a) 
    &= \bigvee \{(h\star g)(f(x)) \mid x \in \LC M_1 , \, x \leq a \}\\
    &= \bigvee \{(h \circ g)(f(x)) \mid x \in \LC M_1, \, x \leq a \} & &\text{since $f(x) \in \LC M_2$}\\
    &= \bigvee \{h((g \circ f)(x)) \mid x \in \LC M_1, \, x \leq a \}\\
    &= \bigvee \{h((g \star f)(x)) \mid x \in \LC M_1, \, x \leq a \}&&\text{since $x\in\LC M_1$}\\
    &= (h \star (g \star f))(a). && \qedhere
\end{align*}
\end{proof}

\begin{definition}
    For an MT-algebra $M$, define $1_M:M\to M$ by
    \[
    1_M(a) = \bigvee \{x \in \LC M \mid x \leq a\} \ 
    \mbox{ for each }a\in M. \]
\end{definition}

\begin{lemma}\ \label{lem: identity}
    \begin{enumerate}[ref=\thelemma(\arabic*)]
        \item $1_M$  is a proximity morphism for each MT-algebra $M$. \label{1M}
        \item For each proximity morphism $f:M\to N$ between MT-algebras we have 
    \[
    1_N \star f = f = f \star 1_M.
    \]
    \end{enumerate}
\end{lemma}

\begin{proof}
 (1) By the definition of $1_M$,  $1_M(x) = x$ for each $x \in \LC M$. In particular, $1_M$ is identity on $\O M$, and hence \ref{def: proximitymor 1} holds. In view of \cref{rem: prox equiv}, an argument similar to \cite[Lem.~4.8]{Bezh2010} yields that \ref{def: proximitymor 2} and \cref{lem: eqv morphisms 3} hold. It is also immediate from the definition that \ref{def: proximitymor 4} holds. Thus, $1_M$ is a proximity morphism by \cref{lem: eqv morphisms}.
    
 (2)   Let $a\in M$. Then
    \begin{align*}
    (1_N \star f)(a) &= \bigvee \{1_N(f(x))\mid x\in \LC M, \, x\leq a\}\\
    &= \bigvee \{ f(x) \mid x\in \LC M, \, x\leq a\} & &\text{since $f(x) \in \LC N$}\\
    &= f(a) \\
    &= \bigvee \{f(1_M(x)) \mid x\in \LC M, \, x\leq a\} & &\text{since $x \in \LC M$}\\
    &=(f \star 1_M)(a). && \qedhere
    \end{align*} 
\end{proof}

As an immediate consequence of \cref{lem: composition,lem: identity} we obtain:

\begin{theorem} \label{composition is proximity mt-morphism}
The MT-algebras and proximity morphisms between them form a category $\MTP$ where composition is given by $\star$ and identity morphisms are $1_M$.
\end{theorem}

\begin{proof}
In view of \cref{lem: composition,lem: identity}, we only need to check that if $f:M_1\to M_2$ and ${g:M_2\to M_3}$ are proximity morphisms, then so is $g \star f : M_1 \to M_3$. For this it is sufficient to verify (P1)--(P4).

(P1) For $u \in \O M_1$, we have $(g \star f)(u)= (g \circ f)(u)$. Thus, $(g \star f) | _{\O M_1}$ is a frame morphism.
    
(P2) For $a, b \in M_1$, since $\LC M_1$ is closed under finite meets, we have
\begin{align*}
(g \star f)(a) \wedge (g \star f)(b) &= \bigvee \{g(f(x)) \mid x \in \LC M_1, \, x \leq a \} \wedge \bigvee \{g(f(y)) \mid y \in \LC M_1, \, y \leq b \}\\
&= \bigvee \{g(f(x)) \wedge g(f(y)) \mid x , y \in \LC M_1, \, x \leq a, \, y \leq b \} \\
&= \bigvee \{g(f(x \wedge y)) \mid x,y \in \LC M_1, \, x \leq a, \, y \leq b \} \\
&= \bigvee \{g(f(z)) \mid z \in \LC M_1, \, z \leq a\wedge b \}\\
&= (g \star f)(a \wedge b). 
\end{align*}

(P3) Let $S\subseteq\LC M$. By (P2), $g \star f$ is order preserving. Thus, $\bigvee \{ (g \star f)(s) \mid s \in S \} \le (g \star f)(\bigvee S)$. For the reverse inequality, since $(g \star f)(a)\le (g\circ f)(a)$ for each $a \in M_1$ and $f,g$ are proximity morphisms, we obtain
\[
(g \star f)\left(\bigvee S\right) \le (g \circ f)\left(\bigvee S\right) = g\left(\bigvee f[S]\right) = \bigvee \{ g(f(s)) \mid s \in S \} = \bigvee \{ (g \star f)(s) \mid s \in S \}.
\]

(P4) For $a\in M_1$, we have 
\begin{align*}
(g\star f) (a) &=\bigvee \{g(f(x))\mid x\in \LC M_1, \, x\leq a\}\\
&=\bigvee \{ (g \star f)(x)\mid x\in \LC M_1, \, x\leq a\} &&\text{since $x\in\LC M_1$}. && \qedhere
\end{align*}
\end{proof}

Not surprisingly, isomorphims in $\MTP$ are not structure-preserving bijections: 

\begin{example}\label{r:isonotbijection}
 Let $\square$ be the identity on the two-element boolean algebra $2$. Then $(2,\square)$ is an MT-algebra. 
 Also, let
 $M=\{0,a,b,1\}$ be the four-element boolean algebra. Then $(M,\square)$ is an MT-algebra, where $\square : M\to M$ is defined by 
 \[
 \square a =
 \begin{cases}
     0 & \mbox{ if } a=0 \\
     1 & \mbox{ otherwise.}
 \end{cases}
 \]
Observe that $1_M = \Box$ and $1_2$ is the identity on $2$. Since $2 \subseteq M$, we may view $1_M$ as a proximity morphism $f : M \to 2$ and $1_2$ as a proximity morphism $g : 2\to M$. We then have $g \star f=1_M$ and $f \star g = 1_2$. Thus, $g$ is the inverse of $f$ in $\MTP$, and hence $f$ is an isomorphism in $\MTP$. However, $f$ is clearly not a structure-preserving bijection.
\end{example}

In \cref{prop: proximity isos}, we will characterize isomorphisms in $\MTP$, from which we derive that isomorphisms between $T_{D}$-algebras are indeed structure-preserving bijections (observe that $M$ in the above example is not a $T_D$-algebra). This requires more machinery, which we turn to next.

\section{Funayama envelope}\label{sec: Funayama}

In this section, we define the Funayama envelope of a frame $L$ to be the MacNeille completion of its boolean envelope. We show that the resulting assignment $L\mapsto \mathcal{F} L$ maps each frame to a $T_D$-algebra whose opens are isomorphic to $L$, and thus think of $\mathcal{F} L$ as the $T_D$-hull of $L$. We prove that this assignment extends to a categorical equivalence between $\Frm$, $\MTP$, and the full subcategory $\TDMTP$ of $\MTP$ consisting of $T_D$-algebras. The equivalence of the last two categories is explained by the fact that isomorphisms in $\MTP$ are not structure-preserving bijections. We show that this unusual phenomenon disappears in $\TDMTP$.

\subsection{Boolean envelope of a frame}

Recalling the categories $\IA$ and $\IAC$ (\cref{def: IA and Int}), we have:

\begin{definition}\label{defn: essential algebras}\
\begin{enumerate}[label=\upshape(\arabic*), ref = \thedefinition(\arabic*)]
\item An interior algebra $A$ is {\em essential} if the least boolean subalgebra of $A$ generated by  $\sfO A$ coincides with $A$.\footnote{Esakia introduced essential algebras under the name of skeletal algebras (see \cite[Def.~2.5.6]{Esa2019}). Since the name ``skeletal" is overused in topology, we prefer the name essential. This is justified by the fact that we can think of $A$ as an essential extension of $\O A$ in that for each congruence $\Theta$ of the interior algebra $A$, if $\Theta$ is not identity then neither is $\Theta\cap(\O A \times \O A)$.}

\item Let $\SAC$ be the full subcategory of $\IAC$ consisting of essential algebras, and define $\SA$ similarly (as a full subcategory of $\IA$).
\end{enumerate}
\end{definition} 

Clearly $\SA$ is a wide subcategory of $\SAC$. These two categories are closely related to the following categories:

\begin{definition}\label{defn: Heyting albegra}\
    \begin{enumerate}[ref=\thedefinition(\arabic*)]
        \item Let $\HA$ be the category of Heyting algebras and Heyting homomorphisms, and let $\BA$ be the full subcategory of $\HA$ consisting of boolean algebras. \label{def: HA}
        \item Let $\DLat$ be the category of bounded distributive lattices and bounded lattice homomorphisms, and let $\HLat$ be the full subcategory of $\DLat$ consisting of Heyting algebras. \label{def: DLat}
    \end{enumerate}
\end{definition}
Clearly, $\HA$ is a wide subcategory of $\HLat$. To connect these two categories with $\SA$ and $\SAC$, we recall the definition of the boolean envelope of a distributive lattice (see, e.g., \cite[Sec.~V.4]{balbes74}), which is the reflector $\B : \DLat \to \BA$.  

The {\em boolean envelope} or {\em free boolean extension} of a bounded distributive lattice $L$ is a pair $(\mathcal{B}L, e)$, where $\mathcal{B}L$ is a boolean algebra and $e: L \to \mathcal{B}L$ is a bounded lattice embedding satisfying the following universal mapping property: for any boolean algebra $A$ and a bounded lattice homomorphism $h: L \to A$, there is a unique boolean homomorphism $\mathcal{B}h: \mathcal{B}L \to A$ such that $\mathcal{B}h\circ e = h$; i.e., the following diagram commutes:
\[\begin{tikzcd}
L \ar[r, "e"] \ar[d,"h"'] & \B L \ar[dl,dashed,"\B h"]\\
A
\end{tikzcd}\]

We identify $L$ with it image $e[L]$ and treat $L$ as a bounded sublattice of $\B L$ which generates $\B L$.
If $L$ is a Heyting algebra, then the embedding $L\hookrightarrow\B L$ has a right adjoint $\square : \B L\to L$ and $(\B L, \square)$ is an essential interior algebra (see, e.g., \cite[Sec.~2.5]{Esa2019}). Moreover, each bounded lattice homomorphism $h:L_1\to L_2$ lifts uniquely to a continuous morphism $\B h:\B L_1\to\B L_2$. Furthermore, $h$ is a Heyting homomorphism iff $\B h$ is a morphism of interior algebras (see, e.g., \cite[Sec.~2.2]{BMM08}). We thus obtain:

\begin{theorem}\ \label{thm: skeletal and Heyting}
    \begin{enumerate}[ref=\thetheorem(\arabic*)]
        \item $\SA$ is a coreflective subcategory of $\IA$ that is equivalent to $\HA$. \label[theorem]{Ess coreflective}
        \item $\SAC$ is a coreflective subcategory of $\IAC$ that is equivalent to $\HLat$. \label [theorem]{Essc coreflective}
    \end{enumerate}
\end{theorem}

\begin{proof}
    For (1) see \cite[Thm.~2.5.11]{Esa2019}, and (2) is proved similarly (see, e.g., \cite[Thm.~2.14]{BMM08}).
\end{proof}

We next restrict the equivalence in \cref{Essc coreflective} to constructible algebras.

\begin{definition}\label{constructible algebras}\
   \begin{enumerate}
 \item We call $A$ {\em constructible} if 
it is essential and 
$\O A$ is a frame. \label[definition]{constructible algebras 2} 
\item A continuous morphism $f:A\to B$ between constructible algebras  is a \emph{constructible morphism} if $f|_{\O A}: \O A \to \O B$ is a frame morphism.
\item Let $\Cons$ be the category of constructible  algebras and constructible morphisms. 
\end{enumerate} 
\end{definition}

Note that $\Cons$ is a non-full subcatgeory of $\SAC$ since not every bounded lattice homomorphism between frames is a frame morphism. However, isomorphisms in $\Cons$ are isomorphisms in $\SAC$. 
We have the following consequence of \cref{Essc coreflective}:

\begin{theorem}\label{thm: Frm = Cons}
    $\Frm$ is equivalent to $\Cons$.
\end{theorem}

\begin{proof}
    For a Heyting algebra $L$, we have that $L$ is a frame iff $(\B L,\square)$ is a constructible algebra. Indeed, if $L$ is a frame, then $(\B L,\square)$ is an essential interior algebra for which $\O\B L$ is a frame since $\O\B L = L$ (recall that we identify $L$ with $e[L]$). Thus, $(\B L,\square)$ is a constructible algebra. For the same reason, if $(\B L,\square)$ is constructible then $L$ must be a frame. 
    Moreover, for a bounded lattice homomorphism $h:L_1 \to L_2$ between frames, since $\O\B L_i = L_i$ ($i=1,2$) and $\B h|_{L_1} = h$, we have that $h$ is a frame morphism iff $\B h$ is a constructible morphism. Thus, since isomorphisms in $\Cons$ are isomorphisms in $\SAC$, the equivalence of \cref{Essc coreflective}
    restricts to an equivalence between $\Frm$ and $\Cons$. 
\end{proof}

\subsection{Funayama envelope of a frame}

\begin{proposition}
    $\O : \MTP \to \Frm$ is a functor.
\end{proposition}

\begin{proof}
    By \cref{thm: O functor}, $\O M$ is a frame; by (P1), if $f:M\to N$ is a proximity morphism, $\O f:\O M\to\O N$ is a frame morphism. Since $\O M\subseteq\LC M$, the restriction of $1_M$ is the identity on $\O M$, so $\O(1_M) = 1_{\O M}$. For the same reason, if $f:M_1 \to M_2$ and $g:M_2 \to M_3$ are proximity morphisms then 
    \[
    \O(g \star f)=(g \star f )|_{\O M_1}= (g \circ f)|_{\O M_1 }=g|_{\O M_2 } \circ f|_{\O M_1 }=\O g \circ \O f.
    \]
    Thus, $\O : \MTP \to \Frm$ is a functor. 
    \end{proof}

We next show that $\O$ is an equivalence by describing its quasi-inverse using Funayama's result \cite{funayama59} that there is a frame embedding of each frame $L$ into a complete boolean algebra $B$. There are two standard constructions of $B$, which can be realized as the booleanization of the frame of nuclei of $L$ \cite{Joh1982} or as the MacNeille completion of the boolean envelope of $L$ \cite{Gra2011}. As was shown in \cite{BGJ2013}, the two realizations yield the same object up to isomorphism. 

For a frame $L$, let $\overline{\mathcal{B}L}$ be the MacNeille completion of its boolean envelope. We lift the interior operator $\square:\mathcal{B}L\to \mathcal{B}L$ to $\overline{\square}:\overline{\mathcal{B}L}\to \overline{\mathcal{B}L}$ by
\[
\overline{\square} a = \bigvee\{\square b\mid b\in \mathcal{B}L \mbox{ and } b\leq a\}. 
\] 
Then $\left(\overline{\mathcal{B}L},\overline{\square}\right)$ is an MT-algebra such that $\O \mathcal{F} L \cong L$ (see e.g., \cite[p.~8]{BezhanishviliRR2023}).  

\begin{definition}
    For a frame $L$, we call the MT-algebra $\left(\overline{\mathcal{B}L},\overline{\square}\right)$ the {\em Funayama envelope} of $L$ and denote it by $\mathcal{F} L$.
\end{definition}

Our aim is to show that the Funayama envelope extends to a functor $\mathcal F : \Frm \to \MTP$ that is a quasi-inverse of $\O$.

\begin{lemma}\label{frame to MT morphism}
    Each frame morphism $h:L_1 \to L_2$ extends to a proximity morphism $\mathcal{F}h:\mathcal{F}L_1 \to \mathcal{F}L_2$ given by 
    \begin{eqnarray*}
    \mathcal{F}h(a) &=& \bigvee \{\mathcal{B}h(x) \mid x \in \LC\mathcal{F}L_1, \, x \leq a\} \\
    &=& \bigvee \{\mathcal{B}h(b) \mid b \in \mathcal{B}L_1, \, b \leq a\}.
    \end{eqnarray*}
\end{lemma}

\begin{proof}
    By identifying $L_1$ with its image in $\mathcal B L_1$, we have $\O\mathcal{F}L_1 = L_1$, and similarly for $L_2$. Since $\LC\mathcal{F}L_1 \subseteq \mathcal{B}L_1$ and each element of $\mathcal{B}L_1$ is a finite join from $\LC\mathcal{F}L_1$, we have
    \[
    \bigvee \{\mathcal{B}h(b) \mid b \in \mathcal{B}L_1, \, b \leq a\} = \bigvee \{\mathcal{B}h(x) \mid x \in \LC\mathcal{F}L_1, \, x \leq a\}.
    \]
    We show that $\mathcal{F}h$ satisfies (P1)--(P4) of \cref{def: proximitymor}.
     Clearly $\mathcal{F}h| _{\mathcal{B}L_1} = \mathcal{B}h$. In particular, $\mathcal{F}h| _{L_1} =h$, and so (P1) holds. 
     By \cite[Lem.~4.8]{Bezh2010}, $\mathcal{F}h(a \wedge b)=\mathcal{F}h(a) \wedge \mathcal{F}h(b)$, and hence (P2) holds. Since $\LC \mathcal{F}L_1 \subseteq \mathcal{B}L_1$ and $\B h$ is a boolean homomorphism, for each finite $S \subseteq \LC \mathcal{F}L_1$, we get \[\mathcal{F}h\left(\bigvee S\right)=\mathcal{B}h\left(\bigvee S\right)=\bigvee\mathcal{B}h[S] =\bigvee\mathcal{F}h[S],\] and thus (P3) holds. Finally, 
     \[
     \mathcal{F}h(a) = \bigvee \{\mathcal{B}h(x) \mid x \in \LC\mathcal{F}L_1, \, x \leq a\} = \bigvee \{\mathcal{F}h(x) \mid x \in \LC\mathcal{F}L_1, \, x \leq a\},
     \]
     and so (P4) holds, yielding that $\mathcal{F}h$ is a proximity morphism.
\end{proof}

 \begin{proposition}
   $\mathcal{F}: \Frm \to \MTP$ is a functor.
    \end{proposition}
 \begin{proof}
     As we saw above, $\mathcal F$ is well defined both on objects and morphisms of $\Frm$. We show that $\mathcal F$ sends identity morphisms to identity moprhisms and preserves composition. Let $a\in \mathcal{F} L$. Since $\B(1_L)=1_{\B L}$,  we obtain 
         \[\mathcal{F}(1_{L})(a) = \bigvee \{\mathcal{B}(1_L)(x) \mid x \in \LC L_), \, x \leq a\} = \bigvee \{ x \in \LC L_1 \mid x \leq a\} = 1_{\mathcal{F} L}(a).\] Therefore, $\mathcal{F}(1_{L})=1_{\mathcal{F} L}$. Now, let $f:L_1\to L_2$ and $g:L_2\to L_3$ be frame morphisms. Then 
     \begin{align*}
     (\mathcal{F}g \star \mathcal{F}f)(a) 
          &= \bigvee \{\mathcal{F}g( \mathcal{F}f(x)) \mid x \in \LC \mathcal{F}L_1, \, x \leq a \}\\
     &= \bigvee \{\mathcal{F}g( \mathcal{F}f(b)) \mid b \in \mathcal{B}L_1, \, b \leq a \}\\
     &= \bigvee \{\mathcal{F}g(\mathcal{B}f(b)) \mid b \in \mathcal{B}L_1, \, b \leq a \} && (\F f| _{\B L_1}=\B f)\\
     &= \bigvee \{\mathcal{B}g(\mathcal{B}f(b)) \mid b \in \mathcal{B}L_1, \, b \leq a \} && (\F g| _{\B L_2}=\B g; \, b \in \B L_1 \Rightarrow \B f(b) \in \B L_2)\\
     &= \bigvee \{\mathcal{B}(g \circ f)(b) \mid b \in \mathcal{B}L_1, \, b \leq a \}&& (\B g \circ \B f = \B(g \circ f))\\
     &= \mathcal{F}(g \circ f)(a). &&
         \qedhere
     \end{align*}
 \end{proof}

Each frame $L$ is isomorphic to $\O\F L$. We denote the isomorphism by $\rho_L : L \to \O \mathcal{F} L$. By identifying $L$ with $\O \F L$, we view $\rho_L$ as the identity on $L$. In addition, if $M$ is an MT-algebra, then the boolean envelope $\B\O M$ of $\O M$ embeds into $M$ by \cite[Prop.2.5.9]{Esa2019}, and hence so does $\F\O M$ by \cite[Thm.XII.3.4]{balbes74}. We thus identify $\F\O M$ with its image in $M$. 

\begin{lemma}\label{adjoint equivalence}
For $M \in \MTP$, define $\zeta_M:\mathcal{F}\O M \to M$ by 
   \[
   \zeta_M(a)=\bigvee_M \{ x \in \LC M \mid x \leq a \}
   \]
   and $\varphi_M : M \to \mathcal{F}\O M$ by 
   \[
   \varphi_M(b)=\bigvee_{\mathcal{F}\O M}\{x \in \LC M \mid x \leq b\}.
   \] 
   Then $\zeta_M$ and $\varphi_M$ are mutually inverse proximity isomorphisms.
\end{lemma}
\begin{proof}
Since each element of $\cons M$ is a finite join from $\LC M$, we have 
$$\zeta_M(a) = \bigvee_M \{ x \in \LC M \mid x \leq a \} =\bigvee_M \{ b \in \cons M \mid b \leq a \}.$$ 
Thus, it satisfies \ref{def: proximitymor 4}. Since $\zeta_M$ is identity on both $\LC M$ and $\O M$, it also satisfies \ref{def: proximitymor 3} and \ref{def: proximitymor 1}. Finally, it satisfies \ref{def: proximitymor 2} by \cite[Lem.~4.8]{Bezh2010}. Therefore, $\zeta_M$ is a proximity morphism. That $\varphi_M$ is a proximity morphism is proved similarly. It is left to show that $\zeta_M$ and $\varphi_M$ are mutually inverse in $\MTP$. Since ${\zeta_M(x) = \varphi_M(x) = x}$ for each $x\in \LC M$, for $a \in M$, we have 
        \begin{eqnarray*}
        (\zeta_M \star \varphi_M) (a) &=& \bigvee_M \{\zeta_M(\varphi_M(x)) \mid x \in \LC M , \, x \leq a \} \\
        &=& \bigvee_M \{x \in \LC M \mid x \leq a\} = 1_{M}(a);
        \end{eqnarray*}
    and for $b \in \mathcal{F}\O M$, we have 
    \begin{eqnarray*}
    (\varphi_M \star \zeta_M) (b) &=& \bigvee_{\mathcal{F}\O M} \{\varphi_M(\zeta_M(x)) \mid x \in \LC M, \, x \leq b \} \\
    &=& \bigvee_{\mathcal{F}\O M} \{x \in \LC M \mid x \leq b\} = 1_{\mathcal{F}\O M}(b), 
    \end{eqnarray*}
    concluding the proof.
\end{proof}

\begin{remark}
    In general, the $\MTP$-isomorphisms produced in the above result are not bijections. For example, consider the MT-algebra $M = \{0,a,b,1\}$ of \cref{r:isonotbijection}. We have $\mathcal{F}\O M=\{0,1\}$, and so $(\zeta_M\star\varphi_M)(a)=0\neq a$. This behavior occurs as composition in $\MTP$  is given by $\star$ rather than by usual composition of functions, and identities in $\MTP$ are not identity maps.
\end{remark}

\begin{lemma}\label{natural transformation}
$ $
\begin{enumerate}
    \item $\rho: 1_{\Frm} \to \O  \mathcal{F}$ is a natural transformation.
    \item $\zeta : \mathcal{F}  \O \to 1_{\MTP}$ is a natural transformation.
\end{enumerate}
    
\end{lemma}
\begin{proof}
(1) Let $f:L_1\to L_2$ be a frame morphism. We must show that the following diagram commutes.
\[\begin{tikzcd}[column sep=5em]
    L_1 \ar[r, "f"] \ar[d, "\rho_{L_1}" ']& L_2 \ar[d, "\rho_{L_2}"]\\
    \mathcal{O}\F L_1 \ar[r, "\mathcal{O}\mathcal{F}f"']  &  \mathcal{O}\F L_2 
    \end{tikzcd}\]
As before, we identify $L$ with $\rho_{L}[L]$ and assume that $L\subseteq \mathcal{F} L$. Since the functor $\O$ sends a proximity morphism to its restriction to the frame of opens, commutativity of the diagram amounts to showing that $\mathcal{F}f(a)=f(a)$ for each $a\in L_1$, which follows from the definition of $\mathcal{F}f$. 

(2) Let $g:M_1 \to M_2$ be a proximity morphism between MT-algebras. We must show that the following diagram commutes. 
\[\begin{tikzcd}[column sep=5em]
    M_1 \ar[r, "g"] & M_2 \\
    \mathcal{F}\O M_1 \ar[r, "\mathcal{F}\mathcal{O}g"'] \ar[u, "\zeta_{M_1}"] &  \mathcal{F}\O M_2 \ar[u, "\zeta_{M_2}"']
    \end{tikzcd}\]
First let $x \in \LC M_1$. Then $g(x) \in \LC M_2$ by \cref{lem: lc to lc}. Therefore, $\mathcal{F}\O g(x)=g(x)$, and hence 
     \begin{align*}
        \zeta_{M_2} (\mathcal{F}\O g(x)) = \zeta_{M_2}(g(x)) 
        = \bigvee \{y \in \LC M_2 \mid y \leq g(x) \} 
               = g(x) = g(\zeta_{M_1}(x)),
    \end{align*}
    where the last equality holds since $\zeta_{M_1}(x)=x$. Now let  $a \in \mathcal{F}\O M_1$. Then 
    \begin{align*}
        (\zeta_{M_2} \star \mathcal{F}\O g)(a) &= \bigvee \{\zeta_{M_2}(\mathcal{F}\O g(x)) \mid x \in \LC M_1, \, x \leq a \}\\
        &= \bigvee \{g(\zeta_{M_1}(x)) \mid x \in \LC M_1, \, x \leq a \} = (g \star \zeta_{M_1})(a). \qedhere
    \end{align*}
\end{proof}

\begin{theorem}\label{adjunction of Frm and MT}
The functors $\O$ and $\F$ 
establish an equivalence of $\MTP$ and $\Frm$.
\end{theorem}
\begin{proof}
As we saw in \cref{adjoint equivalence}, the natural transformations $\rho$ and $\zeta$ of \cref{natural transformation} are isomorphisms on all components. Thus, it suffices to show that these are the unit and counit of the adjunction $\mathcal{F}\dashv \O$. 

Let $M \in \MTP$. In view of our identifications, $\O \zeta_{M}$ and $\rho_{\O M}$ are identities. Hence, for $u \in \O M$, we have 
    \[
    \O \zeta_{M} \circ \rho_{\O M}(u)=\O \zeta_M(u)=u.
    \]

Let $L \in \Frm$. For similar reasons, $\rho_{L}$ and $\mathcal{B}\rho_{L}$ are identities. Therefore, for $b \in \mathcal{B}L$,
    \[(\zeta_{\mathcal{F} L} \circ \mathcal{F}\rho_{L})(b)=\zeta_{\mathcal{F} L} (\mathcal{B}\rho_{L}(b))=\zeta_{\mathcal{F} L}(b)=b.\]
    Thus, for $a \in \mathcal{F} L$,
    \[
    (\zeta_{\mathcal{F} L} \circ \mathcal{F}\rho_{L})(a)=\bigvee \{
    \zeta_{\mathcal{F} L}(\mathcal{F}\rho_{L}(b)) \mid  b \in \mathcal{B}L, \, b \leq a \}=\bigvee \{b \in \mathcal{B}L \mid b \leq a \}=a. \qedhere
    \]
\end{proof}

In Example \ref{r:isonotbijection} we have seen that $\MTP$-isomorphisms are not necessarily structure-preserving bijections. The fact that $\mathcal{O}:\MTP\to \Frm$ establishes an equivalence of categories now gives us a characterization of such morphisms (see, e.g., \cite[Prop.~7.47]{AHS2006}).

\begin{proposition}\label{l:isochar}
    Let $f:M\to N$ be a proximity morphism of MT-algebras.
    \begin{enumerate} [ref=\thelemma(\arabic*)]
        \item $f$ is an isomorphism iff $\mathcal{O}f$ is an isomorphism of frames.\label [lemma]{iso}
        \item $f$ is a monomorphism iff $\mathcal{O}f$ is a monomorphism of frames.\label [lemma]{mono}
        \item $f$ is an epimorphism iff $\mathcal{O}f$ is an epimorphism of frames.\label [lemma]{epi}
    \end{enumerate} 
\end{proposition}

Note that, apart from isomorphisms not being bijections between the underlying sets, in $\MTP$ we also have monomorphisms that are not injective and epimorphisms that are not surjective:  
\begin{example}
    In Example \ref{r:isonotbijection}, the maps $f$ and $g$ are both isomorphisms hence both are monic and epic. However, $f$ is not injective and $g$ is not surjective.
\end{example}

This counterintuitive behavior disappears when we restrict our attention to $T_D$-algebras. 

\begin{proposition} {\em \cite[Thm.~6.5]{BezhanishviliRR2023}} \label{lem: fom is td}  
    An MT-algebra $M$ is $T_D$ iff $M \cong \mathcal{F}\O M$.
\end{proposition}

\begin{definition}\label{TD-algebras}
   Let $\TDMTP$ be the full subcategory of $\MTP$ consisting of $T_D$-algebras. 
\end{definition}
 
We have the following:

\begin{theorem}\label{WMT_D&Frm}\
\begin{enumerate}
[ref=\thetheorem(\arabic*)]
    \item {\em $\TDMTP$} is equivalent to $ \MTP$.\label[theorem]{TdMTp equivalent to MTp}
    \item $\Frm$ is equivalent to $\TDMTP$.\label[theorem]{TdMTp equivalent to Frm}
\end{enumerate}
\end{theorem}

\begin{proof}
    (1) Let ${\mathcal e}:\TDMTP \to \MTP$ be the inclusion functor. By \cref{adjoint equivalence} and \cite[Prop.~IV.4.2]{MacLane}, we obtain an adjoint equivalence between $\TDMTP$ and $\MTP$ via the functors $\mathcal e$ and $\mathcal{F}\O$, making $\mathcal{F}\O :\MTP \to \TDMTP$  a quasi-inverse of $\mathcal e$.

    (2) Apply \ref{TdMTp equivalent to MTp} and \cref{adjunction of Frm and MT}.
    \end{proof}

The isomorphism $\varphi_M:M\to \mathcal{F}\mathcal{O}M$ may be seen as the $T_D$-reflection of $M$. Indeed, $\mathcal{F}\O M$ is always a $T_D$-algebra, and if $N$ is a $T_D$-algebra and $f:M\to N$ is a proximity morphism, we may define a proximity morphism $\widehat{f}:\mathcal{F}\O M\to N$ by setting $\widehat{f} = f \star \eta_{\mathcal{F}\O M}$. We then have a commutative diagram in $\MTP$:
\[
\begin{tikzcd}
    M
    \ar[d,"f",swap]
    \ar[rr,"\varphi_M"]
    && \mathcal{F}\O M
    \ar[dll,"\widehat{f}"]\\
    N
\end{tikzcd}
\]

By definition, up to isomorphism, $\O\varphi_M=\rho_{\O M}$ is the identity in $\Frm$. Therefore, the $T_D$-reflection does not do anything in $\Frm$. In fact, for frames there is no concept of the $T_D$-reflection since the language of frames is less expressive than that of MT-algebras. 

Since every MT-algebra is isomorphic to its $T_D$-reflection, by considering the inverse of this isomorphism, we see 
that $\TDMTP$ is also a coreflective subcategory of $\MTP$, with the coreflector given by the counit $\zeta$.

We conclude this section by showing that, unlike the situation in $\MTP$, isomorphisms in $\TDMTP$ are structure preserving bijections. 
For this we use the following lemma, which is a consequence of \cite[Thm.~2.5.11]{Esa2019}. 

\begin{lemma}\label{l: lift to boolean}
    For $f:L\to M$ a frame isomorphism,  $\mathcal{B}f:\mathcal{B}L\to \mathcal{B}M$ is a boolean isomorphism. 
    \end{lemma}

\begin{proof}
     By \cref{Ess coreflective}, $\bf {HA}$ is equivalent to $\bf {SA}$.  
    Therefore, Heyting isomorphims $H_1 \to H_2$ correspond to interior algebra isomorphisms $\mathcal{B}H_1\to\mathcal{B}H_2$. But frame isomorphisms between frames are Heyting algebra isomorphisms, and interior algebra isomorphisms are boolean isomorphisms, so the result follows.
    \end{proof}

\begin{proposition}\label{prop: proximity isos}
    A proximity map $f:M\to N$ between $T_D$-algebras is an isomorphism in $\MTP$ iff it is an order-isomorphism.
\end{proposition}
\begin{proof}
    First suppose that $f:N\to M$ is a proximity isomorphism between $T_D$-algebras. By  \cref{iso}, $\O f$ is an isomorphism of frames and  $\mathcal{B}\O f:\mathcal{B}\O M\to \mathcal{B}\O N$ is a boolean isomorphism. Therefore, it can be lifted to an isomorphism between $\mathcal{F}\O M$ and $\mathcal{F}\O N$ (see, e.g., \cite[Thm.~7.41(ii)]{Davey2002}). As $M$ and $N$ are $T_D$-algebras, they are order-isomorphic to $\mathcal{F}\O M$ and $\mathcal{F}\O N$, and since the isomorphism lifting $\mathcal{B}\O f$ preserves arbitrary joins, it must coincide with $\mathcal{F}\O f=f$. 
    
    Conversely, suppose that $f:M\to N$ is an order-isomorphism. Then its inverse $f^{-1}:N\to M$ is an order-isomorphism. Therefore, for $a\in M$, we have
    \begin{align*}
        (f^{-1}\star f)(a)=\bigvee \{f^{-1}(f(x))\mid x\in \LC M, \, x\leq a\}=\bigvee \{x\in \LC M \mid x\leq a\}=1_M(a).
    \end{align*}
    A similar argument yields that $f\star f^{-1}=1_N$. Thus, $f$ is a proximity isomorphism. 
\end{proof}

\begin{remark}\label{rem: TDMTP}
    The category $\TDMTP$ has the following additional pleasant features:
    \begin{enumerate}[ref=\theremark(\arabic*)]
        \item Identities in $\TDMTP$ are identity functions. In fact, an MT-algebra $M$ is $T_D$ iff the identity $1_M$ in $\MTP$ is the identity function. Indeed, 
    \[
    M \mbox{ is } T_D \iff \forall a \in M, \ a=\bigvee \{ x \in \LC M \mid x \le a \} \iff \forall a \in M, \ a = 1_M(a). \qedhere
    \]
    \item The category $\TDMT$ is a wide subcategory of $\TDMTP$. For, if $f:M\to N$ is a $\TDMT$-morphism, it preserves all finite meets and joins by definition. By \cref{thm: O functor}, its restriction $\O f:\O M\to \O N$ is a frame morphism. For $a\in M$, because $M$ is $T_D$, ${a=\bigvee \{x\in \LC M\mid x\leq a\}}$. Since $f$ preserves all joins, $f(a)=\bigvee \{f(x)\mid x\in \LC M,x\leq a\}$, so it is a $\TDMTP$-morphism. \label[remark]{rem: TDMTP 2}
    \end{enumerate}
    
\end{remark}

\cref{fig: rel categories} summarizes the relationship between the categories introduced in this section. The connecting ``arrows'' should be understood as follows:
\begin{itemize}
    \item red two-sided arrows denote categorical equivalence;
    \item solid black hooks denote full embeddings, with reflections and coreflections noted;
    \item dashed black hooks denote non-full embeddings;
    \item blue  hooks denote wide embeddings;
    \item squiggly lines denote same objects but incomparable morphisms.
\end{itemize} 
(The same color coding will be used in the rest of the paper.)
\begin{figure}
    \centering
    \begin{tikzcd}[row sep=scriptsize, column sep=scriptsize]
	&&&&&& \BA \\
	\IA && \SA && \HA \\
	&&&&&& \DLat \\
	{\IAC} && {\SAC} && \HLat \\
	\\
	&& \Cons && \Frm && \MTP && {\TDMTP} \\
    \\
	&&&&&& \MT && {\TDMT}
	\arrow[hook', from=1-7, to=2-5]
	\arrow["\sf{refl}"', hook, from=1-7, to=3-7]
	\arrow[color={blue}, hook', from=2-1, to=4-1]
	\arrow["\sf{corefl}","\ref{Ess coreflective}"', hook', from=2-3, to=2-1]
	\arrow[color={red},"\ref{Ess coreflective}", leftrightarrow, from=2-3, to=2-5]
   \arrow[color={blue}, hook', from=2-3, to=4-3]
	\arrow[color={blue}, hook', from=2-5, to=4-5]
	\arrow["\sf{corefl}","\ref{Essc coreflective}"', hook', from=4-3, to=4-1]
	\arrow["\ref{Essc coreflective}",color={red}, leftrightarrow, from=4-3, to=4-5]
  	\arrow[hook, from=4-5, to=3-7]
	\arrow[dashed, hook, from=6-3, to=4-3]
	\arrow[color={red}, "\ref{thm: Frm = Cons}",leftrightarrow, from=6-3, to=6-5]
   \arrow[dashed, hook, from=6-5, to=4-5]
	\arrow[color={red},"\ref{adjunction of Frm and MT}",leftrightarrow, from=6-5, to=6-7]
   \arrow[color={red},"\ref{TdMTp equivalent to MTp}", leftrightarrow, from=6-7, to=6-9]
   \arrow[squiggly, no head, from=6-7, to=8-7]
	\arrow[color={blue},"\ref{rem: TDMTP 2}"', hook, from=8-9, to=6-9]
	\arrow[hook', from=8-9, to=8-7]
\end{tikzcd}
    \caption{Relationship between categories}
    \label{fig: rel categories}
\end{figure}
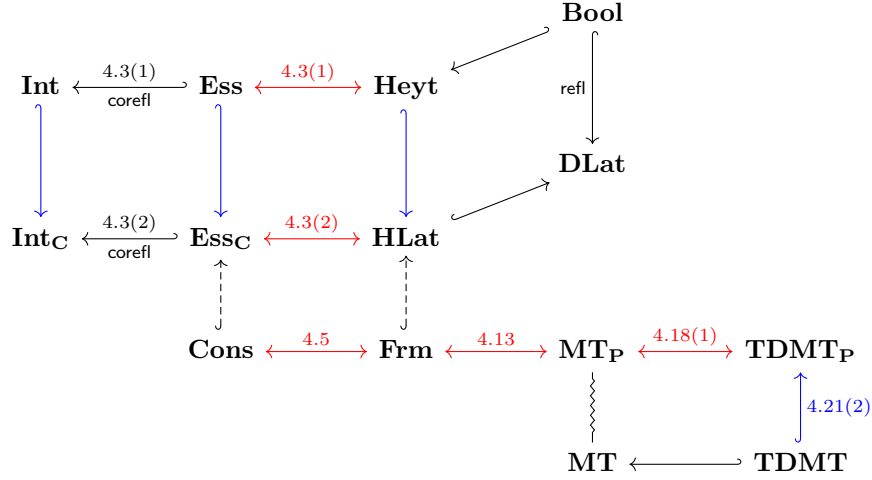

\section{\texorpdfstring{$T_D$}--duality for MT-algebras} \label{sec: TD duality for MT}

 In this section, we generalize the $T_D$-duality of Banaschewski and Pultr \cite{banaschewskitd} to the setting of MT-algebras. This, in particular, yields a generalization of the $T_D$-coreflection \cite[3.7.2]{banaschewskitd} from $T_0$-spaces to arbitrary ones. We argue that the MT setting is more natural for the $T_D$-duality than the frame setting by observing that, unlike the case of $T_D$-spatial frames, the spatial $T_D$-algebras form a full subcategory of $\MT$. 

\subsection{\texorpdfstring{$T_D$}--spectra of MT-algebras}
In this subsection, we introduce the $T_D$-spectrum of an MT-algebra and connect it to the $T_D$-spectrum of a frame. 
\begin{definition}
    For an MT-algebra $M$, let $\at_D M$ be the collection of its locally closed atoms. 
\end{definition}

We view $\at_D M$ as a subspace of the spectrum $\at M$ of $M$ as defined in \cref{Interior algebras and MT-algebras}. To connect $\at_D M$ to $\pt_D \O M$, we recall:

\begin{lemma}\ \label{l:pt(OM) upper bound}
\begin{enumerate}[ref=\thelemma(\arabic*)]
    \item {\em \cite[Prop.~4.8]{BezhanishviliRR2023}} For every MT-algebra $M$, there is a continuous map $\theta:\at M  \to \pt \sfO M$ given by $\theta(x) = {\uparrow} x\cap \sfO M$.\label[lemma]{l:pt(OM) upper bound 1}
    \item {\em \cite[Prop.~4.10]{BezhanishviliRR2023}} If $M$ is a $T_0$-algebra then $\theta$ is a subspace embedding.\label[lemma]{l:pt(OM) upper bound 2}
\end{enumerate} 
\end{lemma}

We show that for $T_0$-algebras, the above embedding yields that $\at_D M$ is hommeomorphic  to $\pt_D \O M$. For this, we use the following:

\begin{lemma}\label{l: atomsint0}
    For a $T_0$-algebra $M$, an element $x\in M$ is an atom iff for each $u\in \sfO M$ we have $x\leq u$ iff $x\nleq \neg u$.
\end{lemma}
\begin{proof}
    Clearly if $x\in M$ is an atom, then the condition in the statement is satisfied. For the converse, if $x\in M$ satisfies the condition, then $x\neq 0$. Let $y\leq x$ with $y\neq 0$. We show that $x\leq y$. Since $M$ is $T_0$, $y=\bigwedge S\wedge \neg v$ for some $S\subseteq \sfO M$ and $v\in \sfO M$. If $x\nleq y$ then either $x\nleq u$ for some $u\in S$  or $x\nleq \neg v$. By assumption, in the former case we get that $x\leq \neg u$, and in the latter case that $x\leq v$. In both cases, $x\leq \bigvee \{\neg u \mid u \in S \} \vee v=\neg y$. Therefore, $y\leq x\leq \neg y$, a contradiction. 
    \end{proof}
Recall (see, e.g., \cite[Sec.~II.3.3]{PicadoPultr2012}) that in a frame $L$, we have a bijection between completely prime filters and prime elements given by $P\mapsto \bigvee (L{\setminus}P)$. 

\begin{lemma} {\em \cite[Prop.~2.6.1, 2.6.2]{banaschewskitd}} \label{l:slicing iff covered}
    Each slicing filter of a frame is  completely prime. Moreover, for a completely prime filter $P$, the following are equivalent:
    \begin{enumerate}
        \item $P$ is a slicing filter;
        \item The corresponding prime is a covered prime;
        \item The corresponding prime is completely meet-irreducible.
    \end{enumerate}
\end{lemma}

We are ready to prove the main result of this subsection.

\begin{theorem}\label{c: td inclusion into at}
    For a $T_0$-algebra $M$, the embedding $\theta:\at M\to \pt \sfO M$ restricts and co-restricts to a homeomorphism $\theta':\at_D M\to \pt_D\sfO M$.  
    \[\begin{tikzcd}
	\at M && \pt \O M \\
	\\
	\at_D M && \pt_D \O M
	\arrow["\theta", from=1-1, to=1-3]
	\arrow[hook,  from=3-1, to=1-1]
    \arrow["\theta'", from=3-1, to=3-3]
	\arrow[hook, from=3-3, to=1-3]
\end{tikzcd}\]
\end{theorem}
\begin{proof}
To see that $\theta'$ is well defined, let $x\in \at_D M$. Since $x\in\LC M$, $x = u \wedge \diamond x$ for some $u \in \O M$ (see footnote \ref{footnote 2}), then $u  \vee \neg \diamond x \in {\uparrow} x\cap \sfO M$ and $\neg \diamond x \notin {\uparrow} x\cap \sfO M$. We show that $\neg \diamond x \lessdot u \vee \neg \diamond x$. Suppose $\neg \diamond x \leq v \leq u \vee \neg \diamond x$ for some $v \in \sfO M$. Since $x$ is an atom, either $x \leq v$ or $x \leq \neg v$. In the former case, $x \vee \neg \diamond x \le v$, so $(u \wedge \diamond x)\vee\neg\diamond x\le v$, and hence $u\vee \neg\diamond x\le v$. In the latter case, $\diamond x \le \neg v$ since $\neg v$ is closed, so $v\le \neg\diamond x$. Thus, ${\uparrow} x\cap \sfO M$ is a slicing filter. 

That $\theta'$ is one-to-one follows from \cref{l:pt(OM) upper bound 2}. To show it is onto, we need to show that every slicing filter $F\subseteq \sfO M$ is of the form ${\uparrow}x\cap \O M$ for some locally closed atom $x$. Let \[x=\bigwedge F \wedge \bigwedge \{\neg a\mid a\notin F\}.\] 

\begin{claim}\label{claim 1}
    $x$ is an atom.
\end{claim}

\begin{proof}[Proof of Claim]
    Since for each $u\in \sfO M$ we have $x\leq u$ or $x\leq \neg u$, if $x\ne 0$, it is an atom by \cref{l: atomsint0}. Thus, it is enough to show that $x\ne 0$. Let $p$ be the covered prime corresponding to $F$. Since $p=\bigvee (L\setminus F)$, we have $x=0$ iff $\bigwedge F\leq p$. Indeed, \[x = 0 \Longrightarrow \bigwedge F \le \neg\bigwedge \{\neg a \mid a \notin F \} = \bigvee \{ a \mid a \notin F\} = p.\] Conversely, 
\begin{eqnarray*}
    \bigwedge F \le p & \Longrightarrow & \bigwedge F \le \bigvee \{ a \mid a \notin F \} \\
    & \Longrightarrow & \bigwedge F \wedge \neg\bigvee \{ a \mid a \notin F \} = 0 \\
    & \Longrightarrow & \bigwedge F \wedge \bigwedge \{ \neg a \mid a \notin F \} = 0 \Longrightarrow x = 0.
\end{eqnarray*}
But $\bigwedge F \le p$ iff $\bigwedge \{u\vee p\mid u\in F\} = p$. 
Since for $A\subseteq\O M$, $\bigwedge A\in \sfO M$ implies $\bigwedge A=\bigwedge_{\sfO M} A$, the last condition is equivalent to $\bigwedge_{\sfO M}\{u\vee p\mid u\in F\}=p$. However, because $p$ is covered, by \cref{l:slicing iff covered} this means that $u\leq p$ for some $u\in F$. The obtained contradiction proves that $x$ is an atom.
\end{proof} 

\begin{claim}\label{claim 2}
    $F = {\uparrow} x\cap \sfO M$.
\end{claim}

\begin{proof}[Proof of Claim]
    Let $u\in \sfO M$. First suppose that $u\in F$. Then $x\le\bigwedge F\le u$, and so $u \in {\uparrow} x\cap \sfO M$. Next suppose that $u\notin F$. Then $x\le \neg u$. By \cref{claim 1}, $x$ is an atom, so $x \not\le u$, and hence $u\notin {\uparrow} x\cap \sfO M$.
\end{proof} 

\begin{claim}
    $x$ is locally closed.
\end{claim} 

\begin{proof}[Proof of Claim]
    By \cref{claim 2}, $F = {\uparrow} x\cap \sfO M$. Since $F$ is slicing, there exist $a,b\in \sfO M$ such that $a \lessdot b$, $a \notin \upset x \cap \sfO M$, and $b \in \upset x \cap \sfO M$. By \cref{claim 1},  $x$ is an atom. This together with $M$ being a $T_0$-algebra yields that $x = \bigwedge S \wedge \neg v$ for some $S \subseteq \sfO M$ and $v \in \sfO M$. For each $u \in S$, we have $a \leq a \vee (b \wedge u) \leq b$. Since $a \lessdot b$, either $a = a \vee (b \wedge u)$ or $b = a \vee (b \wedge u)$. In the former case, $x \leq b \wedge u \leq a$, and hence $a \in \upset x \cap \sfO M$, a contradiction. Therefore, $b = a \vee (b \wedge u)$, and thus 
    \[
    b \wedge \neg a = [a \vee (b \wedge u)] \wedge \neg a = b \wedge u \wedge \neg a \leq u.
    \] 
    Since this is true for each $u\in S$, we obtain $b \wedge \neg a \leq \bigwedge S$. Consequently, 
    \[
    x \le b \wedge \neg a \wedge \neg v \le  \bigwedge S \wedge \neg v \le x,
    \]
   yielding that $x = b \wedge \neg a \wedge \neg v$. Thus, $x$ is locally closed since $b$ is open and $\neg a \wedge \neg v$ is closed.
\end{proof} 

Consequently, $\theta' : \at_D M \to \pt_D \O M$ is a homeomorphism.
\end{proof}

\begin{corollary} \label{cor: objects}
    For a $T_D$-algebra $M$, $\at M$ is homeomorphic to $\pt_D\sfO M$.
\end{corollary}
\begin{proof}
By \cref{c: td inclusion into at}, there is a homeomorphism $\theta':\at_D M\cong \pt_D\sfO M$. Since $M$ is a $T_D$-algebra, $\at_D M=\at M$, yielding the result.
\end{proof}

The assumption in \cref{c: td inclusion into at} that $M$ is a $T_0$-algebra is necessary. (Note that the assumption is used to show that $\theta':\at_D M\to \pt_D \O M$ is onto.)

\begin{example}
    In the MT-algebra $M$ of \cref{r:isonotbijection}, $\upset {a} \cap \sfO M=\{1\}$ is a slicing filter. But $a \in \at M$ is not locally closed because $\LC M=\{0,1\}$. Thus, $\theta'$ is not onto.
   \end{example}

\subsection{\texorpdfstring{$T_D$}--reflection of MT-algebras and \texorpdfstring{$T_D$}--coreflection of topological spaces}\label{sec: TD coreflection}

We now focus our attention on morphisms and look at the MT-analogues of D-morphisms of Banaschewski and Pultr (see \cref{def: D morphism}), which we will also call D-morphisms. Our aim is to show that the spatial $T_D$-algebras form a full reflective subcategory of the category of MT-algebras and D-morphisms, thus yielding a pointfree version of the $T_D$-coreflection of $T_0$-spaces defined in \cite{banaschewskitd}. We emphasize that this $T_D$-reflection is not expressible in the language of frames. 

\begin{definition}\ \label{defn: TopLC}
    \begin{enumerate}
        \item We call a continuous map \emph{locally closed} if it maps locally closed points to locally closed points. 
        \item Let $\TopLC$ be the wide subcategory of $\Top$ whose morphisms are locally closed maps.
    \end{enumerate} 
\end{definition}

We point out that identity maps are locally closed and that the composition of two locally closed maps is locally closed, so $\TopLC$ indeed forms a category. We let $\ToTopLC$ be the full subcategory of $\TopLC$ consisting of $T_0$-spaces\footnote{As follows from \cref{lem: D morphism}, $\ToTopLC$ is precisely the category $_D\Top$ defined in \cite[3.7.2]{banaschewskitd}.}, and note that $\TDTop$ is a full subcategory of $\ToTopLC$ since every continuous map between $T_D$-spaces is automatically locally closed. 

We show that locally closed maps between $T_0$-spaces can be seen as topological duals of D-morphisms. For this we recall the following result from \cite[Prop.~2.7.1]{banaschewskitd}: 

\begin{lemma} 
\label{l: every slicing is neighborhood}
    For a $T_0$-space $X$, every slicing filter of $\Omega X$ is of the form 
    \[
    F_x := \{ U \in \Omega X : x \in U \}
    \] 
    for some locally closed $x\in X$.
\end{lemma}

\begin{proposition} \label{lem: D morphism}
    A continuous map $f:X\to Y$ between $T_0$-spaces is locally closed iff $\Omega f$ is a D-morphism. 
\end{proposition}

\begin{proof}
    Suppose that $f:X\to Y$ is a locally closed map between $T_0$-spaces. Consider a slicing filter of $\Omega X$, which by \cref{l: every slicing is neighborhood} is of the form $F_x$ for some locally closed $x\in X$. Then $f(x)\in Y$ is locally closed. We have 
    \begin{align*}
    (\Omega f)^{-1}(F_x)&=\{U\in \Omega Y\mid \Omega f(U)\in F_x\}
=\{U\in \Omega Y\mid f^{-1}(U)\in F_x\}\\
    &=\{U\in \Omega Y\mid x\in f^{-1}(U)\}=F_{f(x)}.
    \end{align*}
    Since $f(x)$ is locally closed, $F_{f(x)}$ is slicing, and hence $\Omega f$ is a D-morphism. 
    
    For the converse, suppose that $\Omega f$ is a D-morphism. If $x\in X$ is locally closed, the same computation as above shows that $(\Omega f)^{-1}(F_x)=F_{f(x)}$, and because $\Omega f$ is a D-morphism, $F_{f(x)}$ is a slicing filter. Thus, $f(x)$ is locally closed by \cref{l: every slicing is neighborhood}, and hence $f$ is a locally closed map.
   \end{proof}

We next introduce D-morphisms for MT-algebras. 

\begin{samepage}
\begin{definition}\ \label{defn: MTD}
    \begin{enumerate}
        \item An MT-morphism $f$ is a \emph{D-morphism} if the left adjoint $f^*$ maps locally closed atoms to locally closed atoms. 
        \item Let $\MTD$ be the category of MT-algebras and D-morphisms. We also let $\SMTD$ be the full subcategory of $\MTD$ consisting of spatial MT-algebras and $\STOMTD$ the full subcategory of $\SMTD$ consisting of spatial $T_0$-algebras.
    \end{enumerate}
    \end{definition}
\end{samepage}

We point out that identity maps are D-morphisms and that the composition of two D-morphisms is a D-morphism, so $\MTD$ indeed forms a category. 
Also, note that $\STDMT$ is a full subcategory of $\STOMTD$ since each MT-morphism between $T_D$-algebras is automatically a D-morphism (because every atom in a $T_D$-algebra is locally closed). 
The following result holds for arbitrary (not only $T_0$) spaces.

\begin{lemma}\label{l: lc iff dmt}
A continous map $f:X\to Y$ is locally closed iff $f^{-1}:\mathcal{P}Y\to \mathcal{P}X$ is a D-morphism.    
\end{lemma}
\begin{proof}
For each $x\in X$, we have 
\begin{align*}
   (f^{-1})^*(\{x\})&=\bigcap \{S\in \mathcal{P}Y\mid \{x\}\subseteq f^{-1}(S)\}=\bigcap \{S\in \mathcal{P}Y\mid f(x)\in S\}=\{f(x)\}. 
\end{align*}
The result follows, since a point in a space is locally closed iff the corresponding singleton is a locally closed element in the MT-algebra of all subsets. 
\end{proof}
 
\begin{remark} \label{e: very simple nonT0}
    The above result is no longer true if we replace the functor $\P$ with $\Omega$: consider the inclusion $\{0\}\se \{0,1\}$ where both sets are given the trivial topology. The dualization of this map is the identity on the two-element frame, which is a D-morphism. But $\{0\}$ is locally closed in $\{0\}$, and not in $\{0,1\}$. Of course, by \cref{lem: D morphism}, it does remain true for $\Omega$ if the spaces under consideration are $T_0$. 
\end{remark}

As an immediate consequence of Theorems~\ref{MT duality}, \ref{thm: MT duality for T0 and TD 1}, and the above lemma we obtain: 

\begin{theorem}
$\TopLC$ is equivalent to $\SMTD^{\mathrm{op}}$, and $\ToTopLC$ is equivalent to $\STOMTD^{\mathrm{op}}$. 
\end{theorem}

The above equivalences further restrict to give the equivalence of \cref{thm: MT duality for T0 and TD 2}. We thus arrive at the following commutative diagram: 

\[\begin{tikzcd}[row sep=scriptsize, column sep=scriptsize]
	\TopLC && \SMTD^{\mathrm{op}} \\
	\\
	\ToTopLC && \STOMTD^{\mathrm{op}} \\
	\\
	\TDTop && \STDMT^{\mathrm{op}}
	\arrow[color={red}, leftrightarrow, from=1-3, to=1-1]
   \arrow[color={red}, leftrightarrow,from=3-1, to=3-3]
    \arrow[color={red}, leftrightarrow, from=5-1, to=5-3]
   \arrow[hook, from=3-3, to=1-3]
	\arrow[hook, from=5-1, to=3-1]
	\arrow[hook, from=5-3, to=3-3]
    \arrow[hook, from=3-1, to=1-1]
	\end{tikzcd}\]

 We next study the relationship between D-morphisms of MT-algebras and D-morphisms of frames. Recalling \cref{thm: O functor}, we have:

\begin{lemma}\label{l: preimageofneighborhoodfilter}
    For a complete boolean homomorphism $f:M\to N$ between MT-algebras, we have
    \[
    (\O f)^{-1}(\uparrow x\cap \O N)={\uparrow} f^*(x)\cap \O M,
    \]
    for all atoms $x\in N$.
\end{lemma}
\begin{proof}
  By the adjointness property, we have that $x\leq f(a)$ iff $f^*(x)\leq a$ for each $a\in M$. This, by definition, means that for each $a\in\O M$, 
  \[
  a\in (\O f)^{-1}(\uparrow x\cap \O N) \iff x \le f(a) \iff f^*(x) \le a \iff a\in {\uparrow} f^*(x)\cap \O M. \qedhere
  \] 
 \end{proof}

\begin{proposition}\ \label{l: MT between TD is Dmor}
\begin{enumerate}[ref=\theproposition(\arabic*)]
    \item An MT-morphism $f:M\to N$ between $T_0$-algebras is a D-morphism iff $\O f$ is a D-morphism. \label [lemma]{Of is D}
    \item Any MT-morphism $f:M\to N$ between $T_D$-algebras is a D-morphism.\label [lemma]{f is D}
\end{enumerate}
\end{proposition}
\begin{proof}
(1) Let $f:M\to N$ be an MT-morphism between $T_0$-algebras. First suppose that $\O f$ is a D-morphism, and that $x\in N$ is a locally closed atom. By \cref{c: td inclusion into at}, ${\uparrow} x\cap \O M$ is a slicing filter, and hence so is  ${\uparrow} f^*(x)\cap \O M$ by \cref{l: preimageofneighborhoodfilter}. Thus, $f^*(x)$ is locally closed by reapplying \cref{c: td inclusion into at}. 

Next suppose that $f$ is a D-morphism. Let $F\subseteq \O N$ be a slicing filter. By \cref{c: td inclusion into at}, $F = {\uparrow} x\cap \O N$ for some atom $x\in M$, which is locally closed by \cref{c: td inclusion into at}. By assumption, $f^*(x)\in M$ is a locally closed atom. By \cref{c: td inclusion into at}, ${\uparrow} f^*(x)\cap \O M$ is a slicing filter, and hence so is $(\O f)^{-1}({\uparrow} x\cap \O N)$ by \cref{l: preimageofneighborhoodfilter}. Thus, $\O f$ is a D-morphism.

(2) This follows immediately from the fact that in a $T_D$-algebra all atoms are locally closed.     
\end{proof}

By \cite[Sec.~3.7.2]{banaschewskitd}, $T_D$-spaces are a coreflective subcategory of the category of $T_0$-spaces with locally closed maps. The $T_D$-coreflection of a $T_0$-space $X$ is $\pt_D \Omega X \se X$, which up to homeomorphism is the inclusion of locally closed points of $X$ into it. 
We next define a pointfree analogue of this construction, without restricting to $T_0$ objects on either side. 

Define $\chi_M:M\to \mathcal{P}\at_D M$ by 
$$\chi_M(a)=\{x\in \at_D M\mid x\leq a\}.$$ 
Note that $\chi_M(\bigvee_i a_i)=\bigcup_i \chi_M(a_i)$, and so 
$\chi_M[\O M]$ is a subframe of $\P \at_D M$. We will regard $\mathcal{P}\at_D M$ as an MT-algebra whose opens are precisely this subframe. Thus, 
$\chi_M:M\to \mathcal{P}\at_D M$ is an MT-morphism onto a spatial $T_D$-algebra.

\begin{lemma}\label{lem:sigma is mt and D}
   The map $\chi_M:M\to \mathcal{P}\at_D M$ is a D-morphism.
\end{lemma}
\begin{proof}
    As observed above, the map is an MT-morphism. The atoms of $\mathcal{P}\at_D M$ are the singletons. For $x\in \at_D M$, we have $\chi_M^*(\{x\})=x$, which is locally closed. Thus, $\chi_M$ is an $\MTD$-morphism. 
    \end{proof}

\begin{theorem}\label{t: sigmareflection}
The category $\STDMT$ is a full reflective subcategory of $\MTD$. 
\end{theorem}
\begin{proof}
The subcategory is full by \cref{f is D}. For any MT-algebra $M$, by \cref{lem:sigma is mt and D}, the map $\chi_M:M\to \mathcal{P}\at_D M$ is a D-morphism onto a spatial $T_D$-algebra. Suppose that $f:M\to N$ is a D-morphism with $N$ a spatial $T_D$-algebra. Define $\widehat{f}:\mathcal{P}\at_D M\to N$ by $\widehat{f}(S)=\bigvee \{f(x)\mid x\in S\}$. We show that the following diagram commutes: 
\[
\begin{tikzcd}
    M
    \ar[d,"f",swap]
    \ar[rr,"\chi_M"]
    && \mathcal{P}\at_D M
    \ar[dll,"\widehat{f}"]\\
    N
\end{tikzcd}
\]
For $a\in M$, 
$\widehat{f}(\chi_M(a))=\bigvee \{f(x)\mid x\in \at_D M, \, x\leq a\}$. It is clear, then, that $\widehat{f}(\chi_M(a))\leq f(a)$. 
For the other inequality, since $N$ is spatial it suffices to show that $y\leq f(a)$ implies $y\leq \widehat{f}(\chi_M(a))$ for all $y\in \at N$. 
If $y\leq f(a)$, then $f^*(y)\leq a$. By assumption on $f$, 
$f^*(y)\in \at_D M$. Therefore, \[y\leq f(f^*(y))\leq \bigvee \{f(x)\mid x\in \at_D M, \, x\leq a\}.\] Thus, 
$y\le\widehat{f}(\chi_M(a))$, as desired. 
Finally, we show that $\widehat{f}$ is an MT-morphism: 
\begin{align*}
    \widehat{f}\left(\bigcup_i S_i\right)&=\bigvee \{f(x)\mid x\in \bigcup_i S_i\}=\bigvee \bigvee_i \{f(x)\mid x\in S_i\}=\bigvee_i\bigvee \{f(x)\mid x\in S_i\}=\bigvee_i \widehat{f}(S_i);
\end{align*}
it is left to see that $\widehat{f}$ maps opens to opens. However, the commutativity of the diagram gives $\widehat{f}(\chi_M(a))=f(a)$ for $a\in \O M$, and the result follows.
\end{proof}

We conclude this subsection by showing that the coreflection in \cite[Sec.~3.7.2]{banaschewskitd} may be obtained as the dualization of the above reflection. 

\begin{definition}
    For a topological space $X$, let $X_D$ be the subspace of $X$ consisting of locally closed points. 
\end{definition}

Since $x \in X$ is locally closed if and only if $\{x\}$ is locally closed in the MT-algebra $\mathcal{P} X$, we have a homeomorphism $h_D:X_D\cong\at_D \mathcal{P}X$ given by $h_D(x)=\{x\}$. 
From now on, we will identify these spaces.

\begin{lemma}\label{l: dualize sigma}
    The inclusion $i_D:X_D\subseteq X$ is such that $\mathcal{P}(i_D)=\chi_{\mathcal{P}X}$.
\end{lemma}
\begin{proof}
Since $\mathcal{P}(i_D)={i_D}^{-1}$, $\P(i_D)(Y) = Y\cap X_D$ for each $Y\subseteq X$. Therefore, under the identification described above, 
\[
\P(i_D)(Y) = Y\cap X_D=\{\{x\}\in \at_D\P X\mid \{x\}\subseteq Y\}=\chi_{\P X}(Y). \qedhere
\]
\end{proof}

\begin{theorem}
The category $\TDTop$ is a full coreflective subcategory of $\TopLC$. The coreflection is given by the inclusion $X_D\se X$.
\end{theorem}

\begin{proof}
Let $X$ be a space. By definition, the inclusion $X_D\subseteq X$ is a locally closed map. Suppose that $Y$ is a $T_D$-space and $f:Y\to X$ is a locally closed map. By \cref{l: lc iff dmt}, $f^{-1}:\mathcal{P}X\to\mathcal{P}Y$ is a $\MTD$-morphism. By \cref{t: sigmareflection}, there is an $\MTD$-morphism $\widehat{f^{-1}}:\mathcal{P}\at_D\mathcal{P}X\to\mathcal{P}Y$ such that the diagram on the right commutes:
\begin{center}
\begin{minipage}{0.4\textwidth}
\[
\begin{tikzcd}
    X_D
    \ar[rr,"\subseteq"]
    && X\\
    Y
    \ar[urr,"f"]
    \ar[dashed,u,"\widehat{f}"]
\end{tikzcd}
\]
\end{minipage}
\begin{minipage}{0.4\textwidth}
\[
\begin{tikzcd}
    \mathcal{P}X
    \ar[d,"f^{-1}",swap]
    \ar[rr,"\chi_{\mathcal{P}X}"]
    && \mathcal{P}\at_D\mathcal{P}X
    \ar[dll,"\widehat{f^{-1}}"]\\
    \mathcal{P}Y
\end{tikzcd}
\]
\end{minipage}
\end{center}
By \cref{l: dualize sigma,thm: MT duality for T0 and TD 2}, there must be a locally closed map $\widehat{f}:Y\to X_D$ such that the diagram on the left commutes.
\end{proof}

The above theorem yields the $T_D$-coreflection of Banaschewski and Pultr. In \cite[3.7.2]{banaschewskitd} it was described as the embedding $\pt_D\Omega X\to X$ for every $T_0$-space $X$. One of the advantages of our approach is that we do not have to restrict to $T_0$-spaces.

\begin{remark}
As we saw above, the $T_D$-coreflection of a space is neatly captured by dualizing the spatial $T_D$-reflection of $\MTD$. By contrast, the frame setting is not expressive enough for this purpose. Indeed, we recall from \cite[3.7.2]{banaschewskitd} that the $T_D$-spatialization of a frame $L$ is given by 
\begin{align*}
    \sigma_L&:L\to \Omega \pt_D L,\\
    a&\mapsto \{P\in \pt_D L\mid a\in P\}.
\end{align*}
If we dualize $\sigma_L$ using $\pt_D$, we obtain a homeomorphism $\pt_D\Omega \pt_D L \to \pt_D L$, which is a trivial $T_D$-coreflection. On the other hand, if we dualize $\sigma_L$ using $\pt$,  we obtain $\pt \sigma_L:\pt\Omega\pt_D L\to \pt L$. Since $\pt_D L$ is the subspace of locally closed points of $\pt L$, this gives the inclusion of the soberification of $\pt_D L$ into $\pt L$, which is not the $T_D$-coreflection. In fact, the soberification of a $T_D$-space is $T_D$ only in the trivial case where the starting space is both sober and $T_D$. We will explore the interplay between soberification and the $T_D$ axiom in \cref{sec: duality for sober maps}.
\end{remark}

\section{Duality for spatial MT-algebras and proximity morphisms} \label{sec: duality for sober maps}

In this final section, we generalize the duality of \cref{MT duality}
between $\Top$ and $\SMT$ to incorporate proximity morphisms between spatial MT-algebras. This is done by introducing the notion of a sober map, a continuous map from one space to the soberification of another, and by showing that frame morphisms between spatial frames and their corresponding proximity morphisms between spatial MT-algebras are characterized by sober maps. As a corollary, we obtain the topological and MT analogues of the category of $T_D$-spatial frames and frame morphisms.

We begin by recalling that $\Sob$ is a reflective subcategory of $\Top$, and that the reflector $\s:\Top\to\Sob$ is given by the soberification $\pt \Omega$ (see, e.g., \cite[p.~44]{Joh1982}).
The unit $\lambda: 1_{\Top} \to \s$ 
is given by $\lambda_X(x)=F_x$ for each $X \in \Top$ and $x \in X$. 

\begin{definition}\label{defn: sober map}
    For topological spaces $X$ and $Y$, we call a continuous map $f:X\to\s Y$ a \emph{sober map} from $X$ to $Y$, and denote it by $f : X \leadsto Y$.
\end{definition}

If $f : X \leadsto Y$ and $g : Y \leadsto Z$ are two sober maps then their composition $g \bullet f : X \leadsto Z$ is given by $\lambda _{\s Z}^{-1} \circ \s g \circ f : X \to \s Z$, which is well defined since $\lambda_{\s Z}:\s Z\to\s\s Z$ is a homeomorphism. By identifying $\s Z$ with $\s\s Z$, the composition $g \bullet f$ can be described as $\s g \circ f$. By this identification, we have that $\s \lambda_X = \lambda_{\s X}$ is the identity on $\s X$. Consequently, since $\lambda$ is a natural transformation, for each $f : X \leadsto Y$, the following diagram commutes:
\[
\begin{tikzcd}[sep=4em]
    X \ar[r, "f"] \ar[d,"\lambda_X"'] & \s Y \ar[d,"\lambda_{\s Y}"]\\
    \s X \ar[r, "\s f"'] & \s Y = \s\s Y
\end{tikzcd}
\]
Therefore, $f \bullet \lambda_X = \s f \circ \lambda_X = \lambda_{\s Y} \circ f = f$. Similarly, for each $g : Y \leadsto X$, we have that $\lambda_X \bullet g = g$. We thus arrive at the following new category:

\begin{definition}\label{defn: TopS}
    Let $\TopS$ be the category of topological spaces and sober maps between them, where composition is given by $\bullet$ and identity morphisms are $\lambda_X$.
\end{definition}

\begin{remark}\label{isos in TopS}
    A sober map $f : X \leadsto Y$ is a $\TopS$-isomorphism iff $\s f : \s X \to \s Y$ is a homeomorphism. To see this, suppose there is $g : Y \leadsto X$ such that $g \bullet f = \lambda_X$ and $f \bullet g = \lambda_Y$, so $\s g \circ f = \lambda_X$ and $\s f \circ g = \lambda_Y$ (see the left diagram below). By applying $\s$ to the former, $\s(\s g \circ f) = \s \s g \circ \s f = \s \lambda_X$. Therefore, by identifying $\s = \s\s$, we have that $\s g \circ \s f$ is the identity on $\s X$. Similarly, $\s f \circ \s g$ is the identity on $\s Y$, and hence $\s f$ is a homeomorphism.
    \[\begin{tikzcd}[sep=4em]
        X & \s Y\\
        \s X & Y
        \ar[from=1-1, to=1-2, "f"]
        \ar[from=2-2, to=2-1, "g"]
        \ar[from=2-2, to=1-2, "\lambda_Y"']
        \ar[from=1-1, to=2-1, "\lambda_X"']
        \ar[from=1-2, to=2-1, "\s g", shift left]
        \ar[from=2-1, to=1-2, "\s f", shift left]
    \end{tikzcd}
    \qquad
    \begin{tikzcd}[sep=4em]
        X & \s Y\\
        \s X & Y
        \ar[from=1-1, to=1-2, "f"]
        \ar[from=2-2, to=1-2, "\lambda_Y"']
        \ar[from=1-1, to=2-1, "\lambda_X"']
        \ar[from=1-2, to=2-1, "g'", shift left]
        \ar[from=2-1, to=1-2, "\s f", shift left]
    \end{tikzcd}\]
    Conversely, suppose $\s f$ is a homeomorphism (see the right diagram, where we identify $\s = \s\s$). Then there is a continuous map $g' : \s Y \to \s X$ which is inverse to $\s f$. Let $g : Y \leadsto X$ be given by $g = g' \circ \lambda_Y$.  By identifying $\s = \s\s$, we have that $\s g' = g'$ and $ \s \lambda_Y$ is the identity. Therefore, since $g'$ is the inverse of $\s f$,
    \begin{align*}
    g \bullet f 
    = \s g \circ f 
    = \s g' \circ \s \lambda_Y \circ f
    = g' \circ f 
    = g' \circ \s f \circ \lambda_X 
    = \lambda_X.
    \end{align*}
    Similarly, $f \bullet g = \s f \circ g = \s f \circ g' \circ \lambda_Y = \lambda_Y$. Consequently, $g$ is the inverse of $f$, so $f$ is a $\TopS$-isomorphism.
\end{remark}

Our aim is to show that \TopS is dually equivalent to the full subcategory $\SMTP$ of $\MTP$ consisting of spatial MT-algebras. To define a functor from $\SMTP^{\mathrm{op}}$ to \TopS, we need the following lemma, where $\eta$ is the counit of MT-duality (see \cref{MT duality}).

\begin{lemma}\label{proximity to quasi}
    Suppose $f : M \to N$ is a proximity morphism between spatial MT-algebras. Define $\ats f : \at N \leadsto \at M$ by 
    \[
    \ats f(y) = \{ \eta_M(a) \mid a \in \O M, \, y \le f(a) \}
    \]
   for each $y\in\at N$. 
    Then $\ats$ is a sober map.
\end{lemma}

\begin{proof}
    Since $f$ is a proximity morphism, its restriction $f|_{\O M} : \O M \to \O N$ is a frame morphism, so $\pt(f|_{\O M}) : \pt\O N \to \pt\O M$ is a continuous map, as is  {$\theta : \at N \to \pt\O N$}  by \cref{l:pt(OM) upper bound 1}.  Because $M$ is spatial, $\eta_M : \O M \to \Omega\at M$ is an isomorphism, so there is a homeomorphism $\psi := \pt(\eta_M^{-1}) : \pt\O M \to \pt\Omega\at M = \s\at M$. The composition
    \[\begin{tikzcd}[column sep=3.5em]
        \at N \ar[r, "\theta"] & \pt\O N \ar[r, "\pt(f|_{\O M})" ] & \pt\O M \ar[r, "\psi"] & \s\at M 
    \end{tikzcd}\]
   is clearly a sober map. But 
    \begin{align*}
        \psi \circ \pt(f|_{\O M}) \circ \theta(y)
        &= \psi \circ \pt(f|_{\O M})(\{b \in \O N \mid y \leq b\})\\
        &= \psi(\{a \in \O M \mid y \leq f(a) \})\\
        &= \{\eta_M(a) \mid a \in \O M, \, y \leq f(a) \}\\
        &= \ats f(y)
    \end{align*}
    for each $y\in\at N$, completing the proof.
   \end{proof}
   
   Note that the inverse of $\psi : \pt \O M \to \s \at M$ is given by $\s\theta$ (where $\theta:\at M  \to \pt \sfO M$ is defined in \cref{l:pt(OM) upper bound}). Indeed, recalling the counit $\sigma : 1_{\Frm} \to \Omega\circ\pt$ from \cref{frames and coframes}, for $x \in \pt\O M$,
\begin{equation*}
\tag{\ding{170}} \label{eq:s theta psi}
\begin{aligned}
    (\s \theta \circ \psi)(x) 
    &= \s\theta(\eta_M[x])\\
    &= \{U  \in  \Omega \pt \O M \mid \theta^{-1}(U) \in \eta_M[x]\}\\
    &= \{U  \in  \Omega \pt \O M \mid \exists u \in x : \theta^{-1}(U) = \eta_M(u)\}\\
    &= \{U  \in  \Omega \pt \O M \mid \exists u \in x : U = \sigma_{\O M}(u)\}\\
    &= \{U \in \Omega \pt \O M \mid x \in U\}\\
    &= \lambda_{\pt \O M}(x) = x,
\end{aligned}
\end{equation*}
where the last equality is true by identifying $\s = \s\s$ since $\pt \O M$ is sober. Moreover, for each $x \in \at M$,
\begin{align*}
 \psi \circ \theta (x)&=\psi(\upset x \cap \O M)\\
 &= \eta_M [\upset x \cap \O M]\\
 &=\{\eta_M(a) \mid a \in \O M, \, x \leq a\}\\
 &=\lambda_{\at M}.   
\end{align*}
Thus, for each $y \in \s \at M$, 
\[
y=\s(\lambda_{\at M}) (y)=\s(\psi \circ \theta)(y)= \psi \circ \s \theta (y).
\]

\begin{proposition}
     $\ats:{\SMTP}^{\mathrm{op}} \to \TopS$ is a functor. 
     \end{proposition}

\begin{proof}
    For a spatial MT-algebra $M$, let $\ats M = \at M$ and for a proximity morphism $f:M \to N$ between spatial MT-algebras, let $\ats f:\at N \leadsto \at M$ be defined as in \cref{proximity to quasi}. Then $\ats$ is well defined on both objects and morphisms. Moreover, for proximity morphisms $f:M_1 \to M_2$ and $g:M_2 \to M_3$, by \cref{proximity to quasi} and \eqref{eq:s theta psi}, 
   \begin{align*} 
    \ats f \bullet \ats g
    &=\s\ats f \circ \ats g\\
    &=\psi_{M_1} \circ \pt(f|_{\O M_1}) \circ  \s\theta_{M_2} \circ \psi_{M_2} \circ \pt(g|_{\O M_1}) \circ \theta_{M_3}\\ 
    &=\psi_{M_1} \circ \pt(g|_{\O M_2} \circ f|_{\O M_1} ) 
     \circ \theta_{M_3}\\
    &= \psi_{M_1} \circ \pt((g \star f)|_{\O M_1}) \circ \theta_{M_3}\\
    &= \ats(g \star f).
    \end{align*}
    Furthermore, for $y \in \ats M$, 
    \begin{align*} 
    \ats 1_M (y)&=\{\eta_M(a) \mid a \in \O M , \, y \leq 1_M(a)\}\\ 
    &= \{ \eta_M(a) \mid a \in \O M , \, y \leq a\}\\ 
    &=\{\eta_M(a) \mid a \in \O M , \, y \in \eta_M(a)\}\\
    &=F_y =\lambda_{\ats M}(y).
    \end{align*}
    Thus, $\ats: \SMTP^{\mathrm{op}} \to \TopS$  is a well-defined functor. 
\end{proof}

To define a functor in the other direction, we require the following:

\begin{lemma} \label{MT + prox = prox}
    If $f : M_1 \to M_2$ is a proximity morphism and $g : M_2 \to M_3$ is an  MT-morphism, then $g \circ f : M_1 \to M_3$ is a proximity morphism.
\end{lemma}

\begin{proof}
   Since $g$ is an MT-morphism, it satisfies  \ref{def: proximitymor 1}--\ref{def: proximitymor 3}, and hence so does the composition $g \circ f$. For \ref{def: proximitymor 4}, observe that 
    \begin{align*}
        gf(a)
        &= g\left(\bigvee\{f(x) \mid x \in \LC(M_1), \, x \leq a\}\right) \\
        &= \bigvee\{g(f(x)) \mid x \in \LC(M_1), \, x \leq a\} 
    \end{align*}
    for each $a\in M_1$. Thus, $g \circ f$ is a proximity morphism.
\end{proof}

Let $M,N$ be MT-algebras. If $h : \O M \to \O N$ is a frame morphism, then $h$ lifts to a proximity morphism  given by the following composition in $\MTP$:
    \[\begin{tikzcd}
    M \ar[r, "\varphi_M"] & \F\O M \ar[r, "\F h"] & \F\O N \ar[r, "\zeta_N"] & N.
\end{tikzcd}\]
For a topological space $X$, 
\[\Omega(\lambda_X) = \P(\lambda_X)|_{\O \P X} : \O \P \s X \to \O \P X\]
is a frame isomorphism. Therefore, the frame isomorphism $\Omega(\lambda_X)^{-1}$ lifts to a proximity morphism $h_X : \P X \to \P\s X$, which is a proximity isomorphism by \cref{l:isochar}(1). 

\begin{definition} \label{sober proximity iso}
    For a topological space $X$, let $h_X : \P X \to \P \s X$ be the lift of $\Omega(\lambda_X)^{-1}$ described above.
\end{definition}

Recall from \cref{sec: proximity morphisms} that for proximity morphisms $f,g$ and a locally closed element $x$, we have $(g\star f)(x)= gf(x)$.  
Therefore, for $D \in \LC\P X$, by \cref{adjoint equivalence} we have
\begin{equation*}
\tag{$\clubsuit$}\label{eq:h}
\begin{aligned}
    h_X(D) &= \zeta_{\P Y} \circ \F\Omega(\lambda_X)^{-1} \circ \varphi_{\P X}(D) \\
    &= \zeta_{\P Y} \circ \F\Omega(\lambda_X)^{-1}(D)\\
    &= \zeta_{\P Y}(\B\Omega(\lambda_X)^{-1}(D))\\
    &= \B\Omega(\lambda_X)^{-1}(D). 
\end{aligned}
\end{equation*}

\begin{lemma}\label{quasi to proximity}
    If $f : X \leadsto Y$ is a sober map then $\Pp f := \P f \circ h_Y$ is a proximity morphism from $\P Y$ to $\P X$.
\end{lemma}
\begin{proof}
    By definition, $f$ is a continuous map from $X$ to $\s Y$. Thus, $\P f=f^{-1} : \P \s Y \to \P X$ is an MT-morphism. Consequently, $\P f \circ h_Y : \P Y \to \P X$ is a proximity morphism by \cref{MT + prox = prox}. 
    \end{proof}

\begin{lemma}\label{proximity on locally closed elements}
    Let $f:M \to N$ and $g:M \to N$ be proximity morphisms. If $f(u)=g(u)$ for all $u \in \sfO M$, then $f=g$. 
\end{lemma}
\begin{proof} 
Let $x \in \LC M$. Then $x=u \wedge \neg v$ for some $u, v \in \sfO M$. Therefore, 
\[f(x)=f(u \wedge \neg v)=f(u) \wedge \neg f(v) = g(u) \wedge \neg g(v)=g(u \wedge \neg v)=g(x).
\]
    Thus, for $a \in M$, 
    \[
    f(a)=\bigvee \{f(x) \mid x \in \LC M, \, x \leq a\} = \bigvee \{g(x) \mid x \in \LC M, \, x \leq a\}=g(a). 
    \]
    Consequently, $f=g$.
\end{proof}

We point out that for $U \in \Omega Y$, 
\begin{equation*}
\tag{$\spadesuit$}\label{eq:Pp}
\begin{aligned}
\Pp f(U) &= (\P f \circ h_Y)(U) \\
&= f^{-1}(\Omega(\lambda_{Y})^{-1}(U)) \\
&= \{x \in X \mid f(x) \in \Omega(\lambda_Y)^{-1}(U)\}\\
&= \{x \in X \mid U \in f(x)\}.
\end{aligned}
\end{equation*}
This will be used in what follows.

\begin{proposition} \label{prop: functor Pp}
   $\Pp: \TopS \to {\SMTP}^{\mathrm{op}}$ is a functor. 
   \end{proposition}

\begin{proof}
 For $X \in \TopS$, let $\Pp X=\P X$ and for a sober map $f:X \leadsto Y$, let $\Pp f : \Pp Y \to \Pp X$ be defined as in \cref{quasi to proximity}. Then $\Pp$ is well defined both on objects and morphisms. For sober maps $f:X \leadsto Y$ and $g:Y \leadsto Z$, we show that $\Pp(g \bullet f) = \Pp f \star \Pp g$. By \cref{proximity on locally closed elements}, it suffices to show that they agree on open elements.
Let $U \in \O\P Z = \Omega Z$. Then 
 \[
    \Pp(g \bullet f)(U) = \Pp(\s g \circ f)(U) = \P(\s g\circ f)\circ h_Z(U) = \P f \circ \P\s g \circ h_Z(U)
 \]
 and since $\star$ is usual composition on open elements,
  \[
    \Pp f \star \Pp g(U) = \Pp f \circ \Pp g(U) = \P f \circ h_Y \circ \P g \circ h_Z(U).
 \]
 Therefore, it is enough to show that $h_Y \circ \P g(V) = \P(\s g(V))$ for all $V \in \O\P\s Z = \Omega\s Z$. 
Using \eqref{eq:h}, we have 
\begin{align*}
    h_Y \circ \P g(V) &=\Omega\lambda_{Y}^{-1}(g^{-1}(V))
    = (\s g)^{-1}(V) = \P\s g(V),
\end{align*}
where the second equality holds because 
\begin{align*}
    z \in \Omega(\lambda_Y) ((\s g)^{-1}(V)) &\iff \lambda_Y (z) \in (\s g)^{-1}(V) 
    \iff \s g(\lambda_Y(z)) \in V \\
    &\iff g(z) \in V \iff z \in g^{-1}(V).
\end{align*}
Finally, for $W \in \O \P X$, by 
\eqref{eq:Pp},
 \[\Pp \lambda_X(W)=\{x \in X \mid W \in \lambda_X(x)\} = \{x \in X \mid x \in W\} = W = \Omega(1_{\P X})(W).\] 
 Thus, by \cref{proximity on locally closed elements},  $ \Pp \lambda_X=1_{\P X}$. 
 \end{proof}

We next connect $\Top$ with $\TopS$ and $\MT$ with $\MTP$.

\begin{proposition}
\begin{enumerate}[ref=\theproposition(\arabic*)]
    \item[]
    \item $\Lambda : \Top \to \TopS$ is a functor given by $\Lambda X = X$ for each topological space $X$ and ${\Lambda f = \lambda_Y \circ f}$ for each continuous map $f : X \to Y$. \label[proposition]{functor from top to tops}
    \item $\Gamma: \MT \to \MTP$ is a functor given by $\Gamma M=M $ for each MT-algebra $M$ and $\Gamma g = 1_N \circ g$ for each MT-morphism $g:M \to N$. \label[proposition]{functor from mt to mtp}
    \end{enumerate}
\end{proposition}
\begin{proof}
(1) It is sufficient to show that $\Lambda$ preserves composition and identities. The latter is immediate since $\Lambda 1_X = \lambda_X \circ 1_X = \lambda_X$. For composition, let $f : X \to Y$ and $g : Y \to Z$ be continuous maps.
Then
    \[
        \Lambda g \bullet \Lambda f = (\lambda_Z \circ g) \bullet (\lambda_Y \circ f) =  \s\lambda_Z \circ \s g \circ \lambda_Y \circ f =\s g \circ \lambda_Y \circ f = \lambda_Z \circ g \circ f = \Lambda(g \circ f),
    \]
    where the third equality holds because $\s \lambda_Z$ is the identity and the fourth because $\lambda$ is a natural transformation.
    \[\begin{tikzcd}
        X \ar[r,"f"] \ar[d, "\lambda_X"] & Y \ar[r,"g"] \ar[d, "\lambda_Y"]& Z \ar[d, "\lambda_Z"]\\
        \s X  \ar[r, "\s f"] & \s Y \ar[r, "\s g"] & \s Z \arrow[loop right]{r}{\s \lambda_Z}
    \end{tikzcd}\]
   
     (2) Again, it is sufficient to show that $\Gamma$ preserves composition and identities. For an MT-algebra $M$, let $I_M$ be the identity in $\MT$ and $1_M$ the identity in $\MTP$. Then $\Gamma I_M=1_M \circ I_M= 1_M$. Let $f:M_1 \to M_2$ and $g:M_2 \to M_3$ be MT-mophisms. Then, for $a \in M_1$, 
    \[
    \Gamma(g \circ f)(a)=1_{M_3}(g(f(a)))=\bigvee \{g(f(x) \mid x \in \LC M, \, x \leq a \} = (g\star f)(a).
    \]
    Thus, $\Gamma$ is a functor.
\end{proof}

\begin{lemma} \label{lem: con diagram}    For a continuous map $f : X \to \s Y$, the following diagram commutes:
    \[
\begin{tikzcd}
    X \ar[r, "\varepsilon_X"] \ar[d,"f"'] & \at\P X \ar[r, "\lambda_{\at\!\P\! X}"] \ar[d,"\at\!\P\! f"] & \s\at\P X \ar[d,"\s\!\at\!\P\!f"]\\ 
    \s Y \ar[r, "\varepsilon_{\s\!Y}"'] & \at\P\s Y \ar[r, "\lambda_{\at\!\P\!\s\!Y}"'] & \s\at\P\s Y
\end{tikzcd}
\]
\end{lemma}

\begin{proof}
The left square commutes because $\varepsilon : 1_{\Top} \to \at \P$ is a natural transformation (see \cref{MT duality}). The right square commutes because applying the functor $\at \P$ to the natural transformation $\lambda : 1_{\Top} \to \s$ yields a natural transformation $\lambda \circ (\at  \P) : \at  \P \to \s \at  \P$. 
\end{proof}

\begin{samepage}
\begin{lemma} \label{lemma: atP and atsPp}
    Let $f : X \to \s Y$ be a continuous map. 
    \begin{enumerate}
        \item $\ats\Pp f(x) = \at\P f(x)$ for all $x \in \at \P X$. 
        \item $\s\ats\Pp f = \s\at\P f$.
    \end{enumerate}
\end{lemma}
\end{samepage}

\begin{proof}
    Since (2) follows from (1), it is sufficient to prove (1). Let $x \in \at \P X$. Then $x = \varepsilon(x') = \{x'\}$ for a unique $x' \in X$. Thus, $\at\P f(x) = \varepsilon f(x')$.
    Moreover, $\Pp f(a) = f^{-1}\Omega(\lambda_{\s Y})^{-1}(a) = f^{-1}(a)$ since $\lambda_{\s Y}$ is the identity on $\s Y$. Hence,
    \begin{align*}
        \ats\Pp f(x) &= \{\eta_M(a) \mid a \in \sfO \P \s Y,\ x \leq \Pp f(a)\}\\
        &= \{\eta_M(a) \mid a \in \sfO \P \s Y,\ x \leq f^{-1}(a)\}\\
        &= \{\eta_M(a) \mid a \in \sfO\P \s Y,\ \{x'\} \subseteq f^{-1}(a)\}\\
        &= \{\eta_M(a) \mid a \in \sfO\P \s Y,\ f(x') \in a\}\\
        &= \varepsilon f(x').
    \end{align*}
    Consequently, $\ats\Pp f(x) = \varepsilon f(x') = \at\P f(x)$.
   \end{proof}

\begin{theorem}\label{SMTp equivaent to Tops}
    $\TopS$ is equivalent to $\SMTP^{\mathrm{op}}$.
\end{theorem}

\begin{proof}
     We first define $\widehat \varepsilon : 1_{\TopS} \to \ats \Pp$ by setting $\widehat{\varepsilon}_X =  \lambda_{\at\P X} \circ \varepsilon_X$ for each $X \in \TopS$. By \cref{functor from top to tops}, $\widehat{\varepsilon}_X = \Lambda\varepsilon_X : X \leadsto \at\P X$ is a $\TopS$-isomorphism since $\varepsilon_X$ is a homeomorphism (see \cref{isos in TopS}). We show that $\widehat \varepsilon$ is a natural transformation by showing that the following diagram on the left commutes in $\TopS$. Using the identification $\s = \s\s$, this is equivalent to showing that the diagram on the right commutes in $\Top$:
\[
\begin{tikzcd}
    X \ar[r, "\widehat \varepsilon_X", rightsquigarrow] \ar[d, "f"', rightsquigarrow] & \at \P X \ar[d, "{\ats}\Pp f", rightsquigarrow] \\
    Y \ar[r, "\widehat \varepsilon_Y"', rightsquigarrow] & \at\P Y
\end{tikzcd}
\qquad
\begin{tikzcd}
    X \ar[d, "f"'] \ar[r, "\varepsilon_X"] & \at\P X \ar[r, "\lambda_{\at\!\P\! X}"] & \s\at\P X \ar[d, "
    \s\!\ats \!\Pp\! f"]\\
    \s Y \ar[r, "\varepsilon_{\s\!Y}"'] & \at\P \s Y \ar[r, "\lambda_{\at\!\P\! \s \!Y}"'] & \s\at\P \s Y
\end{tikzcd}
\]
By \cref{lem: con diagram}, it suffices to show that $\s\at\P f = \s\ats\Pp f$, which is given by \cref{lemma: atP and atsPp}.

We next define $\widehat \eta : 1_{{\SMTP}^{\mathrm{op}}} \to \Pp \ats$ by setting $\widehat \eta_M = \eta_M \circ 1_M$ for each $M \in \SMTP$. By \cref{functor from mt to mtp}, $\widehat \eta_M = \Gamma \eta_M : M \to \P\at M$ is a proximity isomorphism since $\eta_M$ is an isomorphism of MT-algebras 
(because $M$ is spatial; see \cref{MT duality}). We show that $\widehat \eta$ is a natural transformation by showing that the following diagram commutes in $\SMTP$:
\[
\begin{tikzcd}
    M \ar[r, "\widehat \eta_M", rightarrow] \ar[d, "g"', rightarrow] & \P\at M \ar[d, "{\Pp}\!\ats\! g", rightarrow] \\
    N \ar[r, "\widehat \eta_N"', rightarrow] & \P\at N
\end{tikzcd}
\]
By \cref{proximity on locally closed elements}, it is enough to show that the diagram commutes on open elements. 
Let $u \in \sfO M$. By \eqref{eq:Pp},
\begin{align*}
\Pp\ats g \star \widehat{\eta}_M(u)
 &= \Pp\ats g \circ \widehat{\eta}_M(u)\\
&=\Pp\ats g(\eta_M(u))\\
&= \{y \in \at N \mid \eta_M(u)\in \ats g(y)\}.
\end{align*}
Moreover, 
\begin{align*}
\eta_N \star g(u)
&=\eta_N \circ g(u) = \{y \in \at N \mid y \leq g(u)\}.
\end{align*}
Thus, it is enough to recall from \cref{proximity to quasi} that $\ats g(y) = \{\eta_M(u) \mid u \in \O M, \, y \leq g(u)\}$.
Hence, $\TopS$ is equivalent to $\SMTP^{\mathrm{op}}$. 
\end{proof}

Let $\TDTopS$ be the full subcategory of $\TopS$, and let $\STDMTP$ be the full subcategory of $\SMTP$ consisting of $T_D$-algebras. We have:

\begin{corollary}\label{TDTopS}
    $\TDTopS$ is equivalent to $\STDMTP^{\mathrm{op}}$.
\end{corollary}
\begin{proof}
    By \cref{thm: MT duality for T0 and TD 2}, $X \in \TDTopS$ implies $\Pp X=\P X \in \STDMTP^{\mathrm{op}}$, and $M \in \STDMTP^{\mathrm{op}}$ implies $\ats M=\at M \in \TDTopS$. Thus, the equivalence of \cref{SMTp equivaent to Tops} restricts to an equivalence between $\TDTopS$ and $\STDMTP^{\mathrm{op}}$. 
\end{proof}

Let $\TDSFrm$ be the full subcategory of $\Frm$ consisting of $T_D$-spatial frames. The equivalence of \cref{TdMTp equivalent to Frm} restricts to yield:

\begin{proposition}\label{TDFrm}
  $\TDSFrm$ is equivalent to {\em $\STDMTP$}. 
\end{proposition}
\begin{proof}
    It suffices to show that $L \in \TDSFrm$ implies $\mathcal{F} L \in \bf \STDMTP$, and  $M \in \STDMTP$ implies $ \O M \in \TDSFrm$. 
    If $L \in \TDSFrm$ then there exists a $T_D$-space $X$ such that $L \cong \Omega X$. 
    Since $X$ is a $T_D$-space, $\P X$ is a $T_D$-algebra by \cref{thm: MT duality for T0 and TD 2}, and hence $\mathcal{F} \Omega X \cong \P X$ by \cref{lem: fom is td}. 
    Thus, $\mathcal{F} L \cong \mathcal{F} \Omega X \cong \P X$ is a spatial $T_D$-algebra, and hence $\mathcal{F} L \in \STDMTP$. If $M \in \STDMTP$ then $M \cong \P X$ for some $T_D$-space $X$, and hence $\O M = \Omega X \in \TDSFrm$. 
    \end{proof}

Putting together \cref{SMTp equivaent to Tops,TDTopS,TDFrm}, we arrive at the following commutative diagram: 
\[
\begin{tikzcd}[sep=4em]
    \TopS & \SMTP^{\mathrm{op}} \\
    \TDTopS & \STDMTP^{\mathrm{op}} & \TDSFrm^{\mathrm{op}} & \\
    \arrow[color={red},"\cong", <->, from=1-1, to=1-2]
    \arrow[color={red},"\cong", <->, from=2-1, to=2-2]
    \arrow[color={red},"\cong", <->, from=2-2, to=2-3]
    \arrow["\mathrm{full}", hook, from=2-1, to=1-1]
    \arrow["\mathrm{full}", hook, from=2-2, to=1-2]
\end{tikzcd}
\]
\section*{Tables of relevant categories}
For the reader's convenience, we conclude by listing the categories considered in this article, indicating the objects, morphisms, and where the categories appear in the body of the text. 

\begin{table}[H]
    \centering
    \begin{tabular}{|l|l|l|l|}
        \hline
       {\bf Category} & {\bf Objects} & {\bf Morphisms} & \\
       \hline\hline
     {\Top} & topological spaces &  continuous maps & \cref{frames and coframes}\\\hline
      {\Sob} & sober spaces &  continuous maps & \cref{frames and coframes}\\\hline
            {$\ToTop$} & $T_0$-spaces &  continuous maps & \cref{Interior algebras and MT-algebras}\\\hline
            {$\TDTop$} & $T_D$-spaces &  continuous maps & \cref{Interior algebras and MT-algebras}\\\hline
            {$\TopLC$} & topological spaces &  locally closed maps & \cref{defn: TopLC}\\\hline
            {$\ToTopLC$} & topological spaces &  locally closed maps &\cref{sec: TD coreflection}\\ \hline
            {$\TopS$} & topological spaces &  sober maps &\cref{defn: TopS}\\\hline
            {$\TDTopS$} & TD-spaces &  sober maps &\cref{TDTopS}\\\hline
\end{tabular}
    \caption{Categories of topological spaces}
\end{table}

\begin{table}[H]
    \centering
    \begin{tabular}{|l|l|l|l|}
        \hline
       {\bf Category} & {\bf Objects} & {\bf Morphisms} & \\ 
       \hline\hline
        {\Frm} & frames & frame morphisms &\cref{frames and coframes}\\ \hline
        {\FrmD} & frames & D-morphisms & \cref{frames and coframes}\\ \hline
        {\SFrm} & spatial frames & frame morphisms & \cref{frames and coframes}\\ \hline
        {\TDSFrm} & $T_D$-spatial frames & frame morphisms & \cref{frames and coframes}\\ \hline
        {\TDSFrmD} & $T_D$-spatial frames & D-morphisms & \cref{frames and coframes}\\ \hline
\end{tabular}
    \caption{Categories of frames}
\end{table}

\begin{table}[H]
    \centering
    \begin{tabular}{|l|l|l|l|}
        \hline
        {\bf Category} & {\bf Objects} & {\bf Morphisms} & \\ 
        \hline\hline
        {\MT} & MT-algebras & MT-morphisms & \cref{defn: MT algebra}\\ \hline
         {\MTP} & MT-algebras &  proximity morphisms & \cref{composition is proximity mt-morphism}\\\hline
         {\MTD} &  MT-algebras &  D-morphisms & \cref{defn: MTD}\\\hline
         {\TDMT} & $T_D$-algebras &  MT-morphisms & \cref{def: T0 and TD algebras}\\\hline
          {\TDMTP} & $T_D$-algebras &  proximity morphisms & \cref{TD-algebras}\\\hline
           {\ToMT} & $T_0$-algebras &  MT-morphisms & \cref{def: T0 and TD algebras}\\\hline
           \hline
           {\SMT} & spatial MT-algebras & MT-morphisms & \cref{Interior algebras and MT-algebras}\\ \hline
         {\SMTD} & spatial  MT-algebras &  D-morphisms & \cref{defn: MTD}\\\hline
       {\SMTP} &  spatial MT-algebras &  proximity morphisms & \cref{sec: duality for sober maps}\\\hline
         {\STDMT} & spatial $T_D$-algebras &  MT-morphisms & \cref{Interior algebras and MT-algebras}\\\hline
          {\STDMTP} & spatial $T_D$-algebras &  proximity morphisms & 
          \cref{TDTopS}\\\hline
           {\STOMT} & spatial $T_0$-algebras &  MT-morphisms & \cref{Interior algebras and MT-algebras}\\\hline
{\STOMTD} & spatial $T_0$-algebras &  D-morphisms & \cref{defn: MTD}\\\hline
    \end{tabular}
    \caption{Categories of MT-algebras}
\end{table}

\begin{table}[H]
    \centering
    \begin{tabular}{|l|l|l|l|}
        \hline
       {\bf Category} & {\bf Objects} & {\bf Morphisms} & \\ \hline\hline
        {\BA} & boolean algebras & boolean homomorphisms &\cref{defn: Heyting albegra}\\ \hline
       {\DLat} & bdd distr lattices & bdd lattice homomorphisms & \cref{defn: Heyting albegra}\\ \hline
       {\HLat} & Heyting algebras & bdd lattice homomorphisms &\cref{defn: Heyting albegra}\\ \hline
       {\HA} & Heyting algebras & Heyting homomorphisms & \cref{defn: Heyting albegra}\\ \hline
      {\IA} & interior algebras & int alg morphisms &\cref{def: IA and Int}\\ \hline
       {\SA} & essential algebras & int alg morphisms &\cref{defn: essential algebras}\\ \hline
       {\IAC} & interior algebras & continuous morphisms &\cref{def: IA and Int}\\ \hline
        {\SAC} & essential algebras & continuous morphisms& \cref{defn: essential algebras}\\ \hline
         {\Cons} & constructible algebras & constructible morphisms& \cref{constructible algebras}\\ \hline
       \end{tabular}
    \caption{Other categories}
\end{table}

\bibliographystyle{alpha}

\begin{thebibliography}{BMM08}

\bibitem[AHS06]{AHS2006}
J.~Ad\'amek, H.~Herrlich, and G.~E. Strecker.
\newblock Abstract and concrete categories: the joy of cats.
\newblock {\em Repr. Theory Appl. Categ.}, (17):1--507, 2006.
\newblock Reprint of the 1990 original [Wiley, New York].

\bibitem[AT62]{aull62}
C.~E. Aull and W.~J. Thron.
\newblock Separation axioms between {$T\sb{0}$} and {$T\sb{1}$}.
\newblock {\em Indag. Math.}, 24:26--37, 1962.

\bibitem[BBI16]{BBI2016}
G.~Bezhanishvili, N.~Bezhanishvili, and R.~Iemhoff.
\newblock Stable canonical rules.
\newblock {\em J. Symb. Log.}, 81(1):284--315, 2016.

\bibitem[BD74]{balbes74}
R.~Balbes and P.~Dwinger.
\newblock {\em Distributive lattices}.
\newblock University of Missouri Press, 1974.

\bibitem[Bez10]{Bezh2010}
G.~Bezhanishvili.
\newblock Stone duality and {G}leason covers through de {V}ries duality.
\newblock {\em Topology Appl.}, 157(6):1064--1080, 2010.

\bibitem[Bez12]{Bezh2012}
G.~Bezhanishvili.
\newblock De {V}ries algebras and compact regular frames.
\newblock {\em Appl. Categ. Structures}, 20(6):569--582, 2012.

\bibitem[BGJ13]{BGJ2013}
G.~Bezhanishvili, D.~Gabelaia, and M.~Jibladze.
\newblock Funayama's theorem revisited.
\newblock {\em Algebra Universalis}, 70(3):271--286, 2013.

\bibitem[BH14]{BezhJohn2014}
G.~Bezhanishvili and J.~Harding.
\newblock Proximity frames and regularization.
\newblock {\em Appl. Categ. Structures}, 22(1):43--78, 2014.

\bibitem[BMM08]{BMM08}
G.~Bezhanishvili, R.~Mines, and P.~J. Morandi.
\newblock Topo-canonical completions of closure algebras and {H}eyting algebras.
\newblock {\em Algebra Universalis}, 58(1):1--34, 2008.

\bibitem[BP10]{banaschewskitd}
B.~Banaschewski and A.~Pultr.
\newblock Pointfree aspects of the {$T_D$} axiom of classical topology.
\newblock {\em Quaest. Math.}, 33(3):369--385, 2010.

\bibitem[BR23]{BezhanishviliRR2023}
G.~Bezhanishvili and R.~Raviprakash.
\newblock Mc{K}insey-{T}arski algebras: an alternative pointfree approach to topology.
\newblock {\em Topology Appl.}, 339:Paper No. 108689, 30 pages, 2023.

\bibitem[CZ97]{CZ1997}
A.~Chagrov and M.~Zakharyaschev.
\newblock {\em Modal logic}, volume~35 of {\em Oxford Logic Guides}.
\newblock The Clarendon Press, Oxford University Press, New York, 1997.

\bibitem[DP02]{Davey2002}
B.~A. Davey and H.~A. Priestley.
\newblock {\em Introduction to lattices and order}.
\newblock Cambridge University Press, New York, second edition, 2002.

\bibitem[dV62]{deV62}
H.~de~Vries.
\newblock {\em Compact spaces and compactifications. {A}n algebraic approach}.
\newblock PhD thesis, University of Amsterdam, 1962.

\bibitem[Esa19]{Esa2019}
L.~L. Esakia.
\newblock {\em Heyting algebras}, volume~50 of {\em Trends in Logic---Studia Logica Library}.
\newblock Springer, Cham, 2019.
\newblock Edited by G. Bezhanishvili and W. H. Holliday, Translated from the Russian by A. Evseev.

\bibitem[Fun59]{funayama59}
N.~Funayama.
\newblock Imbedding infinitely distributive lattices completely isomorphically into {B}oolean algebras.
\newblock {\em Nagoya Math. J.}, 15:71--81, 1959.

\bibitem[Ghi10]{Ghilardi2010}
S.~Ghilardi.
\newblock Continuity, freeness, and filtrations.
\newblock {\em J. Appl. Non-Classical Logics}, 20(3):193--217, 2010.

\bibitem[Gr{\"a}11]{Gra2011}
G.~Gr{\"a}tzer.
\newblock {\em Lattice theory: foundation}.
\newblock Birkh\"auser/Springer Basel AG, Basel, 2011.

\bibitem[Har77]{Har1977}
R.~Hartshorne.
\newblock {\em Algebraic geometry}, volume No. 52 of {\em Graduate Texts in Mathematics}.
\newblock Springer-Verlag, New York-Heidelberg, 1977.

\bibitem[Joh82]{Joh1982}
P.~T. Johnstone.
\newblock {\em Stone spaces}, volume~3 of {\em Cambridge Studies in Advanced Mathematics}.
\newblock Cambridge University Press, Cambridge, 1982.

\bibitem[Kur22]{Kur1922}
C.~Kuratowski.
\newblock Sur l'op{\'e}ration {{\(A\)}} de l'analysis situs.
\newblock {\em Fund. Math.}, 3:182--199, 1922.

\bibitem[ML98]{MacLane}
S.~Mac~Lane.
\newblock {\em Categories for the working mathematician}, volume~5 of {\em Graduate Texts in Mathematics}.
\newblock Springer-Verlag, New York, second edition, 1998.

\bibitem[MT44]{MT1944}
J.~C.~C. McKinsey and A.~Tarski.
\newblock The algebra of topology.
\newblock {\em Ann. of Math. (2)}, 45:141--191, 1944.

\bibitem[Nat90]{Naturman1990}
C.~A. Naturman.
\newblock {\em Interior algebras and topology}.
\newblock PhD thesis, University of Cape Town, 1990.

\bibitem[PP12]{PicadoPultr2012}
J.~Picado and A.~Pultr.
\newblock {\em Frames and locales}.
\newblock Frontiers in Mathematics. Birkh\"auser/Springer Basel AG, Basel, 2012.

\bibitem[PP21]{PP2021}
J.~Picado and A.~Pultr.
\newblock {\em Separation in point-free topology}.
\newblock Birkh\"auser/Springer, Cham, 2021.

\bibitem[RS63]{RS1963}
H.~Rasiowa and R.~Sikorski.
\newblock {\em The mathematics of metamathematics}, volume~41 of {\em Mathematical Monographs}.
\newblock Pa\'nstwowe Wydawnictwo Naukowe, Warsaw, 1963.

\end{thebibliography}

\makeatletter
\begin{multicols}{2}
\enddoc@text
\end{multicols}
\let\enddoc@text\relax
\makeatother
\end{document}